\documentclass[a4paper, 10pt]{article}
\usepackage[utf8]{inputenc}
\usepackage[left=20 mm, bottom=20 mm, right=20 mm, top=20 mm]{geometry}
\usepackage{amsmath, amssymb, amsthm, csquotes, tikz, float}
\usepackage{hyperref}

\usepackage[english]{babel}
\usepackage{graphicx,epstopdf,epsfig}
\usepackage{amsfonts,epsfig,fancyhdr, graphics, hyperref,amsmath, tikz,amssymb}
\usepackage{geometry}
\usepackage{tikz}
\usetikzlibrary{decorations.markings}
\newtheorem{theorem}{Theorem}[section]
\newtheorem{lemma}{Lemma}[section]
\newtheorem{remark}{Remark}[section]

\newtheorem{defn}{Definition}[section]

\newtheorem{eg}{Example}[section]
\numberwithin{equation}{section}

\newcommand{\Names}{Partha Rana, Sriparna Bandopadhyay}
\newcommand{\Title}{On the Consistency of Combinatorially Symmetric Sign Patterns and the Class of $2$-Consistent Sign Patterns}

\def\det{{\rm det}}
\def\sgn{{\rm sgn}}
\def\In{{\rm In}}
\def\Ch{{\rm Ch}}
\def\min{{\rm min}}
\def\max{{\rm max}}

\def\dist{{\rm dist}}
\def\rank{{\rm rank}}

\begin{document}	
	\title{\Title\thanks{
			Corresponding Author: Partha Rana.}
	\author{Partha\ Rana\thanks{Department of Mathematics, Indian Institute of Technology Guwahati, Guwahati,
			Assam 781039, India.\\
            E-mail addresses: r.partha@iitg.ac.in (P. Rana), sriparna@iitg.ac.in (S. Bandopadhyay).}
\and Sriparna Bandopadhyay .}
}

	\markboth{\Names}{\Title}

\date{}

\maketitle

\begin{abstract}
  A sign pattern is a matrix that has entries from the set $\{+,-,0\}$. An $n\times n$ sign pattern $\mathcal{P}$ is called consistent if every real matrix in its qualitative class has exactly $k$ real eigenvalues and $n-k$ nonreal eigenvalues for some integer $k$, with $1\leq k\leq n$. 
  In [C. A. Eschenbach, F. J. Hall, and Z. Li. Eigenvalue frequency and consistent sign pattern matrices. \textit{Czechoslovak Math. J.}, 44(119):461–479, 1994.], the authors 
  established a necessary condition for irreducible, tridiagonal patterns with a $0$-diagonal to be consistent.  Subsequently, they proposed that this condition is also sufficient for such patterns to be consistent. In this article, we first demonstrate that this proposition does not hold.
    We characterize all irreducible, tridiagonal sign patterns with a $0$-diagonal of order at most five that are consistent. Moreover, we establish useful necessary conditions for irreducible, combinatorially symmetric sign patterns to be consistent. 
  Finally, we introduce the class $\Delta$ of all $2$-consistent sign patterns and provide several necessary conditions for sign patterns to belong to this class.
\end{abstract}


	\maketitle



\section{Introduction}
	A matrix whose entries belong to the set $\{+, -, 0\}$ is called a 
	\textit{sign pattern matrix} (or simply, a \textit{sign pattern} or a 
	\textit{pattern}). The set of all $n\times n$ sign patterns is denoted 
	by $\mathcal{Q}_n$. For $\mathcal{P} = [p_{ij}] \in \mathcal{Q}_n$, the set  
	$$
	\mathcal{Q}(\mathcal{P}) =
	\left\{
	A = [a_{ij}] \in \mathbb{R}^{n\times n} \ \middle| \ 
	\sgn(a_{ij}) = p_{ij} \ \text{for all} \ i,j = 1, 2, \dots, n
	\right\}
	$$
	is called the \textit{qualitative class} of $\mathcal{P}$.
	
	Suppose $P$ is a property that a real matrix may or may not satisfy. 
	A sign pattern $\mathcal{P}$ is said to \textit{require} $P$ if every real matrix in $\mathcal{Q}(\mathcal{P})$ has property $P$ and to \textit{allow} $P$ if some real matrix in $\mathcal{Q}(\mathcal{P})$ has property $P$. 
	A sign pattern $\mathcal{P} \in \mathcal{Q}_n$ is called \textit{sign nonsingular} if it requires nonsingularity. 
	Equivalently, $\mathcal{P}$ is sign nonsingular if in the standard expansion of $\det(\mathcal{P})$, there exists at least one nonzero term, and all nonzero terms have the same sign. If $\mathcal{P}$ requires singularity, that is, all the terms in the expansion of $\det(\mathcal{P})$ are equal to zero, then $\mathcal{P}$ is said to be \textit{sign singular}.
	
	Suppose that $A$ is an $n\times n$ real matrix, define the \textit{eigenvalue frequency} of $A$ as the ordered pair $S_A = (i_r(A), i_c(A))$, where $i_r(A)$ is the number of real eigenvalues of $A$ 
	and $i_c(A)=n-i_r(A)$ is the number of nonreal eigenvalues of $A$.  A sign pattern $\mathcal{P} \in \mathcal{Q}_n$ is said to be 
	\textit{$k$-consistent}, for some integer $k$ with $0 \leq k \leq n$, if $i_r(A)=k$ for all $A \in \mathcal{Q}(\mathcal{P})$.  
	Henceforth, if $\mathcal{P}$ is $k$-consistent, we denote its eigenvalue frequency by $S_{\mathcal{P}} = (k, n-k)$.  
	A sign pattern $\mathcal{P}$ is called \textit{consistent} if it is 
	$k$-consistent for some $k$ with $0 \leq k \leq n$.

	The \textit{inertia} of an $n \times n$ real matrix $A$, denoted by $\In(A)$, is the triple $\In(A) = (i_+(A), i_-(A), i_0(A))$, where $i_+(A)$, $i_-(A)$, and $i_0(A)$ denote the number of eigenvalues of $A$ with positive, negative, and zero real parts, respectively.
    For a sign pattern $\mathcal{P} \in \mathcal{Q}_n$, the inertia of $\mathcal{P}$ is $\In(\mathcal{P}) = \{\In(A) \mid A \in \mathcal{Q}(\mathcal{P})\}$. 
    If $\In(\mathcal{P})$ contains only one element, then $\mathcal{P}$ is said to \textit{require a unique inertia}.
	
	Two sign patterns $\mathcal{P}_1$ and $\mathcal{P}_2$ are said to be \textit{equivalent} if one can be obtained from the other through a combination of permutation similarity, signature similarity, negation, and transposition. Equivalence preserves certain properties; in particular, $\mathcal{P}_1$ requires a unique inertia (respectively, is consistent) if and only if $\mathcal{P}_2$ requires a unique inertia (respectively, is consistent).

	
	Let $A$ be an $n \times n$ matrix. For subsets $\alpha, \beta \subseteq \{1,2,\ldots,n\}$, $A[\alpha,\beta]$ denotes the submatrix of $A$ consisting of rows, columns corresponding to the indices in $\alpha, \beta$, respectively. In particular, $A[\alpha]$ stands for the principal submatrix $A[\alpha,\alpha]$.

	The \textit{signed directed graph} $D = (V, E)$ of $\mathcal{P} = [p_{ij}] \in \mathcal{Q}_n$ is the directed graph with vertex set $V = \{1, 2, \dots, n\}$ and directed edge $(i, j) \in E$ if and only if $p_{ij} \neq 0$. The directed edge $(i,j)$ is associated with a $+$ if $p_{ij}=+$ and associated with a $-$ if $p_{ij}=-$. 
    A product of the form $P = p_{i_1 i_2} p_{i_2 i_3}\cdots p_{i_k i_{k+1}},$ 
where each factor is nonzero and the indices $i_1,i_2,\dots,i_{k+1}$ are all distinct is called a \textit{path} of length $l(P)=k$ from $i_1$ to $i_{k+1}$.
	If $i_1,i_2,\dots,i_{k}$ are all distinct, then $\gamma=p_{i_1i_2}p_{i_2i_3}\cdots p_{i_ki_1}$ is called a \textit{simple cycle}, denoted by $\gamma=(i_1,i_2,...,i_k)$, of length $l(\gamma)=k$.  For $k=1$, the simple cycle $\gamma=p_{i_1i_1}$ is called a \textit{loop}. The simple cycle $\gamma$ is said to be \textit{positive} (respectively, \textit{negative}) if
	$\sgn(\gamma)=(-1)^{k-1} \, p_{i_1 i_2} p_{i_2 i_3} \cdots p_{i_k i_1}$ is positive (respectively, negative). 
    {Therefore, if $l(\gamma)$ is even and the product $p_{i_1i_2}p_{i_2i_3}\cdots p_{i_ki_1}$ is positive (respectively, negative), then $\gamma$ is a negative (respectively, positive) simple cycle. If $l(\gamma)$ is odd and the product $p_{i_1i_2}p_{i_2i_3}\cdots p_{i_ki_1}$ is positive (respectively, negative), then $\gamma$ is a positive (respectively, negative) simple cycle.} 
Suppose that $\gamma_i$'s are mutually vertex disjoint simple cycles with $l(\gamma_i)=k_i$, for $i=1,2,...,t$, then $\Gamma=\gamma_1\gamma_2\cdots\gamma_t$ is called a \textit{composite cycle} of length $l(\Gamma)=\sum_{i=1}^{t}k_i$ and  $\sgn(\Gamma)=\prod_{i=1}^t\sgn(\gamma_i)$.
{Throughout this paper, we assume all the cycles to be simple unless otherwise mentioned.}
	Suppose that $\Gamma_1 = \gamma_1\gamma_2\cdots\gamma_t$, $\Gamma_2$ be two composite cycles in $D$. We define
$$
\Gamma_1 \setminus \Gamma_2 \;=\; \prod \{\, \gamma_i : \gamma_i \in \Gamma_1 \ \text{and} \ \gamma_i \notin \Gamma_2 \},
$$
where each $\gamma_i\in\Gamma_1$ is a simple cycle in $D$.  
Let $S\subseteq V$, then $D \setminus S$ denotes the subgraph of $D$ obtained by deleting the vertices in $S$ together with all the incident edges, {and $D(S)$ denotes the subgraph of $D$ induced by $S$, that is, the subgraph of $D$ with vertex set $S$, together with all the incident edges}.

	A sign pattern $\mathcal{P} = [p_{ij}] \in \mathcal{Q}_n$ is said to be \textit{combinatorially symmetric} if $p_{ij} \neq 0$ if and only if $p_{ji} \neq 0$. The \textit{underlying signed undirected graph} $G$ of $\mathcal{P}$ is the undirected graph whose edges are signed. The edge $\{i,j\}$ is $+$ (respectively, $-$) if and only if $p_{ij}p_{ji}=+$ (respectively, $-$).
	A \textit{path} $P=\{i_1, i_2\} \{i_2, i_3\} \cdots \{i_k, i_{k+1}\}$ of length $l(P) = k$ in $G$ is a sequence of $k$ edges of $G$, where the vertices $i_1, i_2, \dots, i_{k+1}$ are all distinct. For $k \geq 3$, if the vertices $i_1, i_2,..., i_{k}$ are all distinct, then $\{i_1, i_2\} \{i_2, i_3\} \cdots \{i_k, i_{1}\}$ is called a \textit{simple cycle}, denoted by $\mathcal{C} = i_1 i_2 \cdots i_k$, and $l(\mathcal{C)=}k$. Suppose that $\mathcal{C}_i$'s are mutually vertex disjoint simple cycles in $G$, of length $k_i$, for $i=1,2,...,t$, then $\mathcal{C}=\mathcal{C}_1\mathcal{C}_2\cdots \mathcal{C}_t$ is called a \textit{composite cycle} of length $l(\mathcal{C})=\sum_{i=1}^{t}k_i$. A set of nonadjacent edges $\mathcal{M} = \{\{i_1, j_1\}, \{i_2, j_2\}, \dots, \{i_k, j_k\}\}$ is called a \textit{matching} in $G$, of length $l(\mathcal{M})=k$.  
	If $S$ is a set of vertices in $G$, then $G \setminus S$ denotes the subgraph of $G$ obtained by deleting the vertices in $S$ along with their incident edges.
	A combinatorially symmetric sign pattern $\mathcal{P}\in \mathcal{Q}_n$ is called a \textit{tree sign pattern} (respectively, \textit{path sign pattern}) if the underlying signed undirected graph $G$ is a tree (respectively, path). 
	The term path sign pattern is also used to describe a tridiagonal sign pattern.
	
	Let $\mathcal{P} \in \mathcal{Q}_n$ be an irreducible, combinatorially symmetric sign pattern with a $0$-diagonal, and let $G$ denote its underlying signed undirected graph. A path $P$ in $G$ is called a \textit{maximal signed positive path} if it satisfies the following: 
	\begin{itemize}
		\item[i.] the path starts with the first positive edge, 
		or with the first positive edge that follows a negative edge
		\item[ii.] contain successive positive edges and
		\item[iii.] ends at the last positive edge, or when a negative edge occurs. 
	\end{itemize}
	Similarly, a \textit{maximal signed negative path} is defined.
	For example, consider the irreducible, combinatorially symmetric sign pattern $\mathcal{P} = [p_{ij}]$ whose underlying signed undirected graph $G$ is shown in Figure~\ref{nfig11}. 
	\begin{figure}[H] \label{nfig11}
		\tikzset{node distance=1cm}
		\centering
		\begin{tikzpicture}
			\tikzstyle{vertex}=[draw, circle, inner sep=0pt, minimum size=.15cm, fill=black]
			\tikzstyle{edge}=[,thick]
			\node[vertex,label=above:1](v1)at(1,2){};
			\node[vertex,label=right:2](v2)at(2,1){};
			\node[vertex,label=right:3](v3)at(2,-1){};
			\node[vertex,label=below:4](v4)at(1,-2){};
			\node[vertex,label=below:5](v5)at(-1,-2){};
			\node[vertex,label=left:6](v6)at(-2,-1){};
			\node[vertex,label=left:7](v7)at(-2,1){};
			\node[vertex,label=above:8](v8)at(-1,2){};
			\draw[edge](v1)--node[right, yshift=0.2cm, xshift=-0.1cm]{$-$}(v2);
			\draw[edge](v2)--node[right]{$- $}(v3);
			\draw[edge](v3)--node[right, yshift=-0.2cm, xshift=-0.1cm]{$-$}(v4);
			\draw[edge](v4)--node[below]{$+$}(v5);
			\draw[edge](v5)--node[left, yshift=-0.2cm, xshift=0.1cm]{$+$}(v6);
			\draw[edge](v6)--node[left]{$-$}(v7);
			\draw[edge](v7)--node[left, yshift=0.2cm, xshift=0.1cm]{$ +$}(v8);
			\draw[edge](v8)--node[above]{$+$}(v1);
		\end{tikzpicture} \caption{G}\end{figure}
	In this case, the path $\{1,2\}\{2,3\}\{3,4\}$ and $\{6,7\}$ are maximal signed negative path, whereas $\{4,5\}\{5,6\}$ and $\{7,8\}\{8,1\}$ are maximal signed positive path.

	Let $\mathcal{P}$ be an irreducible, combinatorially symmetric sign pattern whose underlying signed undirected graph is $G$ with vertex set $V(G)$. For any two vertices $u, v \in V(G)$, the \textit{distance} $\dist(u,v)$ is the length of the shortest path from $u$ to $v$. If $\mathcal{C} = u_1 u_2 \cdots u_k$ is a cycle in $G$, the distance from a vertex $u \in V(G)$ to $\mathcal{C}$ is given by
	$$\dist(u,\mathcal{C})=\min\{\dist(u,u_i)|~i=1,2,...,k\}.$$

	Suppose $\mathcal{P}=[p_{ij}]\in \mathcal{Q}_n$ is a sign pattern whose underlying signed directed graph $D$ contains a simple $k$-cycle $\gamma$.  Define the $n\times n$ real matrix $B_\gamma(0)=(b_\gamma(0)_{ij})$ by
	\begin{equation} \label{e1}
		b_{\gamma}(0)_{ij}= \begin{cases}
			1 & \text{if}\ p_{ij}=+ ~\text{and}~ \text{is}~ \text{in}~ \gamma \\
			-1 & \text{if}\ p_{ij}=- ~\text{and}~ \text{is}~ \text{in}~ \gamma \\
			0 & \text{elsewhere,}
		\end{cases}
	\end{equation}
	and define the perturbed matrix $B_\gamma(\epsilon)=(b_\gamma(\epsilon)_{ij})$ by
	\begin{equation} \label{e2}
		b_{\gamma}(\epsilon)_{ij}= \begin{cases}
			b_\gamma(0)_{ij} & \text{if}\ p_{ij}~ \text{is}~ \text{in}~ \gamma \\
			\epsilon & \text{if}\ p_{ij}=+ ~\text{and}~ \text{is}~\text{not}~ \text{in}~ \gamma \\
			-\epsilon & \text{if}\ p_{ij}=- ~\text{and}~ \text{is}~ \text{not}~\text{in}~ \gamma \\
			0 & \text{elsewhere,}
		\end{cases}
	\end{equation}
	for some $\epsilon>0$. Since the nonzero eigenvalues of $B_{\gamma}(0)$ are algebraically simple (that is, each has algebraic multiplicity one), 
	thus, for $\epsilon>0$ sufficiently small, the perturbation matrix $B_{\gamma}(\epsilon)$ of $B_{\gamma}(0)$, has $k$ algebraically simple eigenvalues close to the $k$ distinct eigenvalues of $B_{\gamma}(0)$. 
	Consequently, $B_{\gamma}(\epsilon)\in \mathcal{Q}(\mathcal{P})$ has $k$ distinct eigenvalues near the $k$-th complex roots of $1$ or $-1$, depending on the sign of $\gamma$. 

    Let $\Gamma=\gamma_1\gamma_2\cdots \gamma_t$ be a composite cycle in the signed directed graph $D$, where each $\gamma_p$ is a simple cycle of length $l(\gamma_p)$, for $p=1,2,...,t$.  Define the $n\times n$ real matrix $B_\Gamma(0)=(b_\Gamma(0)_{ij})$ by
	\begin{equation} \label{xe3}
		b_{\Gamma}(0)_{ij}= \begin{cases}
			10^p & \text{if}\ p_{ij}=+ ~\text{and is in}~ \gamma_p~\text{for}~p=1,2,...,t \\
			-10^p & \text{if}\ p_{ij}=- ~\text{and is in}~ \gamma_p~\text{for}~p=1,2,...,t\\
			0 & \text{elsewhere.}
		\end{cases}
	\end{equation}
    Then the nonzero eigenvalues of $B_\Gamma(0)$ are the $l(\gamma_p)$-th complex roots of $10^{pl(\gamma_p)}$, $-10^{pl(\gamma_p)}$, depending on the sign of $\gamma_p$, for $p=1,2,...,t$.  Therefore, $B_\Gamma(0)$ has $\sum_{p=1}^{t}l(\gamma_p)$ algebraically simple nonzero eigenvalues. For $\epsilon>0$, define $B_\Gamma(\epsilon)=(b_\Gamma(\epsilon)_{ij})\in \mathcal{Q}(\mathcal{P})$ similarly as in \eqref{e2}, with $\gamma$ replaced by $\Gamma$. Clearly, $B_\Gamma(\epsilon)$ has $\sum_{p=1}^{t}l(\gamma_p)$ algebraically simple nonzero eigenvalues close to the nonzero eigenvalues of $B_\Gamma(0)$ for $\epsilon>0$ sufficiently small. Also, the nonzero real eigenvalues of $B_\Gamma(0)$ of the form $\lambda$, $-\lambda$ cannot merge to form a complex conjugate pair of eigenvalues in $B_\Gamma(\epsilon)$.

Let $\mathcal{P}\in \mathcal{Q}_n$ and let the maximal length of composite cycles in the signed directed graph of $\mathcal{P}$ be $m$. 
For any $B\in \mathcal{Q}(\mathcal{P})$, the characteristic polynomial of $B$, denoted by $\Ch_B(x)$, is of the form
$$
\Ch_B(x) = x^n - E_1(B)x^{n-1} + E_2(B)x^{n-2} - \cdots + (-1)^m E_m(B)x^{n-m},
$$
where $E_k(B)$ is the sum of all cycles (simple or composite) of length $k$ in $B$ properly signed, for all $1 \leq k \leq m$.
Define $V_+(x)$ (respectively, $V_-(x)$) as the number of sign changes in the coefficients of $\Ch_B(x)$ when $x>0$ (respectively, $x<0$).
By Descartes' rule of signs, the number of positive (respectively, negative) real roots of $\Ch_B(x)$ is either exactly equal to $V_+(x)$ 
(respectively, $V_-(x)$) or smaller by an even integer. 
As an example, consider the cubic polynomial
$$
\Ch_B(x) = a_3 x^3 + a_2 x^2 + a_1 x + a_0,
$$
where $B$ is a real matrix and the coefficients $a_i \in \mathbb{R}$, $i=0,1,2,3$, are arbitrary. 
The sign of $\Ch_B(x)$ can be written as
$$
\sgn(\Ch_B(x)) =
\begin{cases}
(\sgn(a_3))(+)\;+\;(\sgn(a_2))(+)\;+\;(\sgn(a_1))(+)\;+\;(\sgn(a_0)), & \text{if~}x>0, \\[1mm]
(\sgn(a_3))(-)\;+\;(\sgn(a_2))(+)\;+\;(\sgn(a_1))(-)\;+\;(\sgn(a_0)), & \text{if~} x<0,
\end{cases}
$$
where $(\sgn(a_i))(\cdot)$ represents the product of $\sgn(a_i)$ and the sign of $x^i$ for $i=3,2,1,0$, 
depending on whether $x>0$ or $x<0$.

In 1994, Eschenbach et al. \cite{1} established several properties of $k$-consistent sign patterns. They also derived a graph theoretic necessary condition for a consistent, irreducible, tridiagonal sign pattern $\mathcal{P}$, based on the number and signs of the edges in the underlying signed undirected graph of $\mathcal{P}$. Furthermore, they proposed that this necessary condition is also sufficient for the consistency of such sign patterns.
In article \cite{2025}, the authors derived several necessary conditions for irreducible, combinatorially symmetric sign patterns to require a unique inertia, based on the number of negative edges and other combinatorial properties of the cycles in their underlying graphs.

In Section~\ref{s2}, we review some relevant results from the literature that will be used in this paper. We then provide a counterexample that shows that the proposition of Eschenbach et al.~\cite{1} does not hold in general.

In Section \ref{s3}, to begin with, we consider irreducible tree sign patterns with a $0$-diagonal. We relate consistent sign patterns in this class with those that require a unique inertia.
 In particular, we prove that for an irreducible, tridiagonal sign pattern $\mathcal{P}$ with a $0$-diagonal of order at most $5$, $\mathcal{P}$ is consistent if and only if it requires a unique inertia. We further derive necessary conditions for consistency of irreducible, tridiagonal sign patterns with a $0$-diagonal and identify a class of such patterns which are not consistent.
In Section~\ref{s4}, we study irreducible, combinatorially symmetric sign patterns whose underlying graphs contain cycles but no loops. We establish several necessary conditions for the consistency of such patterns depending on the sign of the edges and the combinatorial structure of the cycles.

 Eschenbach and Johnson in \cite{1988} proposed the problem of characterizing $k$-consistent sign patterns.
 In \cite{1991}, the authors characterized all $n\times n$ consistent sign patterns $\mathcal{P}$ for which $S_{\mathcal{P}} = (0, n)~\text{or}~ (n, 0)$. Eschenbach \cite{1993a} gave necessary and sufficient conditions for sign patterns of odd order
to have exactly one real eigenvalue.  
In Section \ref{s5}, we examine the class of all sign patterns that require exactly
two real eigenvalues. We denote this class by $\Delta$ and provide useful necessary conditions for sign patterns to
be in $\Delta$. Since a reducible sign pattern is consistent if and only if all its irreducible components are consistent, we consider only irreducible sign patterns in the above study.  In Section \ref{s51}, we also include problems that come as natural extensions of the results obtained in this paper and that pave the way for future research in this topic. 
    
	\section{Preliminaries}\label{s2}

For irreducible, tridiagonal sign patterns with a $0$-diagonal,
Eschenbach et al.~\cite{1} established the following necessary condition for consistency.

\begin{theorem} [Theorem 2.8, \cite{1}] \label{th11}
Let $\mathcal{A}$ be an irreducible, tridiagonal pattern with a $0$-diagonal. 
Then $\mathcal{A}$ is consistent only if the signed undirected graph of $\mathcal{A}$ has, at most, one 
maximal signed path with odd length. 
\end{theorem}
They conjectured that the above necessary condition is also sufficient for an irreducible, tridiagonal sign pattern with a $0$-diagonal to be consistent.
The following example shows that the proposed conjecture is not true. 

\begin{eg} \label{eg2.2}\rm
Let 	$$\mathcal{P}=\begin{bmatrix}
	0&+&0&0&0&0\\ +&0&+&0&0&0\\ 0&+&0&-&0&0\\ 0&0&+&0&+&0\\0&0&0&+&0&+\\0&0&0&0&+&0
\end{bmatrix},$$ be an irreducible, tridiagonal sign pattern with a $0$-diagonal. Therefore, the underlying signed undirected graph of $\mathcal{P}$ has exactly one maximal signed path with odd length. However, any real matrix $B\in \mathcal{Q}(\mathcal{P}),$ is similar to
$$\begin{bmatrix}
	0&1&0&0&0&0\\ a&0&1&0&0&0\\ 0&b&0&-1&0&0\\ 0&0&c&0&1&0\\0&0&0&d&0&1\\0&0&0&0&e&0
\end{bmatrix}, ~\text{where}~ a,b,c,d,e>0.$$ \begin{itemize}
	\item[i.] If $B_1\in  \mathcal{Q}(\mathcal{P})$ is such that $a=b=c=d=e=1$, then $\sigma(B_1)=\{-1.3071 + 0.2151i, -1.3071 - 0.2151i, 1.3071 + 0.2151i, 1.3071 - 0.2151i, 0.5698i, - 0.5698i\}$, so $S_{B_1}=(0,6)$.
	\item[ii.] If $B_2\in  \mathcal{Q}(\mathcal{P})$ is such that $a=b=c=d=1$ and $e=2$, then $\sigma(B_2)=\{-1.5538, 1.5538, -1.4142,\\ 1.4142, -0.6436i, 0.6436i\}$, so $S_{B_2}=(4,2)$.
\end{itemize}  Since $S_{B_1}\neq S_{B_2}$, $\mathcal{P}$ is not consistent. 
\end{eg}

In \cite{2025}, the authors, however, obtained a similar necessary condition for irreducible, tridiagonal sign patterns with a $0$-diagonal to require a unique inertia.
\begin{theorem}[Theorem~2.8, \cite{2025}]\label{th2.51}
Let $\mathcal{P}\in \mathcal{Q}_n$ be an irreducible, tridiagonal sign pattern with a $0$-diagonal. Then $\mathcal{P}$ requires a unique inertia only if the signed undirected graph of $\mathcal{P}$ has, at most, one maximal signed path with odd length.
\end{theorem}

The following necessary and sufficient condition for the consistency of a sign pattern is from \cite{1}.

\begin{remark}\label{xx5}\rm
Let $\mathcal{P}\in \mathcal{Q}_n$. If $n_r(\mathcal{P})$ (respectively, $n_c(\mathcal{P})$) denotes the maximum number of real (respectively, nonreal) eigenvalues allowed by $\mathcal{P}$, then $\mathcal{P}$ is consistent if and only if 
$
n_r(\mathcal{P}) + n_c(\mathcal{P}) = n.
$
\end{remark}

		Let $\mathcal{P}\in \mathcal{Q}_n$ be a tree sign pattern with a $0$-diagonal, with the underlying signed undirected graph $G$. Then $\mathcal{P}_-$ is defined to be the sign pattern whose underlying signed undirected graph $G_-$ is obtained from $G$ by taking the opposite sign of each signed edge of $G$.

	\begin{eg}\rm
    Let $\mathcal{P}\in \mathcal{Q}_n$ be a tree sign pattern with a $0$-diagonal, whose underlying signed undirected graph is $G$  given in Fig.\ref{figx1}.
		\begin{figure}[H]
\centering
\begin{minipage}{.45\textwidth}
\vspace{-1cm}
$$
\mathcal{P}=\begin{bmatrix}
0 & + & 0& 0 & 0\\
- & 0 & + & 0 & 0 \\
0 & + & 0 & + & + \\
0 & 0 & - & 0 & 0\\
0 & 0 & - & 0 & 0
\end{bmatrix}
$$
\end{minipage}
\hspace{0.05\textwidth}
\begin{minipage}{.45\textwidth}
\centering
\tikzset{node distance=1cm}
\begin{tikzpicture}[
  vertex/.style={draw, circle, inner sep=0pt, minimum size=.15cm, fill=black},
  edge/.style={thick}
  ]
  \node[vertex,label=above:1] (v0) at (-3,0) {};
  \node[vertex,label=above:2] (v1) at (-1.5,0) {};
  \node[vertex,label=above:3] (v2) at (0,0) {};
  \node[vertex,label=above:4] (v3) at (1.5,0) {};
  \node[vertex,label=below:5] (v4) at (0,-1) {};
  
  \draw[edge] (v0) -- node[above] {$-$} (v1);
  \draw[edge] (v1) -- node[above] {$+$} (v2);
  \draw[edge] (v2) -- node[above] {$-$} (v3);
  \draw[edge] (v4) -- node[right] {$-$} (v2);
\end{tikzpicture}
\caption{The signed undirected graph $ G$ of $\mathcal{P}$.}
\label{figx1}
\end{minipage}
\end{figure}

	Then the underlying graph corresponding to $\mathcal{P}_-$ is $G_-$ given in Fig.\ref{fign2}.
	\begin{figure}[H]
    \centering
    \begin{minipage}{0.45\textwidth}
    \vspace{-1cm}
        $$
        \mathcal{P}_-=
        \begin{bmatrix}
        0&+&0&0&0\\
           +& 0 & + & 0 & 0 \\
           0& - & 0 & + & + \\
           0& 0 & + & 0 & 0 \\
           0& 0 & + & 0 & 0
        \end{bmatrix}
        $$
    \end{minipage}\hfill
    \begin{minipage}{0.45\textwidth}
        \centering
        \tikzset{
            vertex/.style={draw, circle, inner sep=0pt, minimum size=.15cm, fill=black},
            edge/.style={thick}
        }
        \begin{tikzpicture}
        \node[vertex,label=above:1](v0)at(-3,0){};
            \node[vertex,label=above:2](v1)at(-1.5,0){};
            \node[vertex,label=above:3](v2)at(0,0){};
            \node[vertex,label=above:4](v3)at(1.5,0){};
            \node[vertex,label=below:5](v4)at(0,-1){};
            \draw[edge](v0)--node[above]{$+$}(v1);
            \draw[edge](v1)--node[above]{$-$}(v2);
            \draw[edge](v2)--node[above]{$+$}(v3);
            \draw[edge](v4)--node[right]{$+$}(v2);
        \end{tikzpicture}
        \caption{The signed undirected graph $G_-$ of $\mathcal{P}_-$.}
        \label{fign2}
    \end{minipage}
\end{figure}

\end{eg}

In \cite[Lemma 2.4]{2025}, the authors showed that if $\lambda$ is an eigenvalue of $B=[b_{ij}]\in \mathcal{Q}(\mathcal{P})$ with algebraic multiplicity $m$, where $\mathcal{P}\in \mathcal{Q}_n$ is a tree sign pattern with a $0$-diagonal, then $i\lambda$ is an eigenvalue of $B_-=[b^-_{ij}] \in \mathcal{Q}(\mathcal{P}_-)$ with algebraic multiplicity $m$, where $|b_{ij}|=|b^-_{ij}|$ for all $i,j$.

\begin{remark} \rm \label{th2.3} \cite{2025}
   If $\mathcal{P}\in \mathcal{Q}_n$ is a tree sign pattern with a $0$-diagonal, then $\mathcal{P}$ requires a unique inertia (respectively, consistent) if and only if $\mathcal{P}_-$ is consistent (respectively, requires a unique inertia).
\end{remark}



\section{Consistency of tree sign patterns with a $0$-diagonal and their relation to sign patterns requiring a unique inertia} \label{s3}

In this section, we begin with irreducible path sign patterns with a $0$-diagonal and derive necessary conditions for such patterns to be consistent. We also associate consistent tridiagonal sign patterns with a $0$-diagonal with patterns that require a unique inertia.


\begin{lemma} [Lemma 2.1, \cite{2025}]\label{lem1}
		Suppose that $f(x)$ is a real polynomial of even degree $n$, consisting only of even powers of $x$. If $\lambda$ is a root of $f(x)$ with multiplicity $m$, then $-\lambda$ is also a root of $f(x)$ with multiplicity $m$.
	\end{lemma}



\begin{theorem} \label{th3.1}
If $\mathcal{P}\in \mathcal{Q}_n$ is an irreducible, tridiagonal sign pattern with a $0$-diagonal where $n\leq 5$, then $\mathcal{P}$ is consistent if and only if $\mathcal{P}$  requires a unique inertia.
\end{theorem}	
\begin{proof}
If the number of maximal signed paths of odd length in $\mathcal{P}$ is more than one, then by Theorem \ref{th11} and Theorem \ref{th2.51}, $\mathcal{P}$ is neither consistent nor requires a unique inertia. So, to complete the proof, it is enough to consider $\mathcal{P}$ with the number of maximal signed paths of odd length at most one.

\textbf{Case~1:} $n=2$.\\ The only possible sign patterns up to equivalence are $$\mathcal{P}_1=\begin{bmatrix}
	0&+\\+&0
\end{bmatrix},~\mathcal{P}_2=\begin{bmatrix}
	0&-\\+&0
\end{bmatrix}.$$ Every $B\in \mathcal{Q}(\mathcal{P}_1)$ (respectively, $B\in \mathcal{Q}(\mathcal{P}_2)$) is similar to a real symmetric (respectively, skew-symmetric) matrix. So,  $S_B = (2,0)$ and $\In(B) = (1,1,0)$ for $B \in \mathcal{Q}(\mathcal{P}_1)$, respectively, $S_B = (0,2)$ and $\In(B) = (0,0,2)$ for $B \in \mathcal{Q}(\mathcal{P}_2)$. Therefore, $\mathcal{P}_1, \mathcal{P}_2$ are consistent sign patterns and also require a unique inertia.  


\textbf{Case~2:} $n=3$.\\
Since the number of maximal signed paths of odd length in $\mathcal{P}$ is at most one, the only possible sign patterns up to equivalence are $$\mathcal{P}_1=\begin{bmatrix}
	0&+&0\\+&0&+\\0&+&0
\end{bmatrix}, ~\mathcal{P}_2= \begin{bmatrix}
	0&-&0\\+&0&-\\0&+&0
\end{bmatrix}.$$ Similarly as in case~1, we have $S_B = (3,0)$ and $\In(B) = (1,1,1)$ for $B \in \mathcal{Q}(\mathcal{P}_1)$, respectively, $S_B = (1,2)$ and $\In(B) = (0,0,3)$ for $B \in \mathcal{Q}(\mathcal{P}_2)$. Therefore, $\mathcal{P}_1, \mathcal{P}_2$ are consistent sign patterns and also require a unique inertia.



\textbf{Case~3:} $n=4$.\\
Since the number of maximal signed paths of odd length in $\mathcal{P}$ is at most one, the only possible sign patterns up to equivalence are
$$\mathcal{P}_1=\begin{bmatrix}
	0&+&0&0\\+&0&+&0\\0&+&0&+\\0&0&+&0
\end{bmatrix}, ~\mathcal{P}_2=\begin{bmatrix}
	0&-&0&0\\+&0&-&0\\0&+&0&-\\0&0&+&0
\end{bmatrix}, ~\mathcal{P}_3=\begin{bmatrix}
	0&+&0&0\\+&0&+&0\\0&+&0&-\\0&0&+&0
\end{bmatrix},~\mathcal{P}_4=\begin{bmatrix}
	0&-&0&0\\+&0&-&0\\0&+&0&+\\0&0&+&0
\end{bmatrix}.$$

\begin{itemize}
	\item  Similarly as in case~1, we have $S_B = (4,0)$ and $\In(B) = (2,2,0)$ for $B \in \mathcal{Q}(\mathcal{P}_1)$, respectively, $S_B = (0,4)$ and $\In(B) = (0,0,4)$ for $B \in \mathcal{Q}(\mathcal{P}_2)$. Therefore, $\mathcal{P}_1, \mathcal{P}_2$ are consistent sign patterns and require a unique inertia.
	

	\item 
	Every $B\in \mathcal{Q}(\mathcal{P}_3)$ is similar to
	
	$$B=\begin{bmatrix}
		0&1&0&0\\ a&0&1&0\\ 0&b&0&-1\\ 0&0&c&0
	\end{bmatrix},~\text{where} ~a,b,c>0.$$
	
	Therefore, the characteristic polynomial of $B$ is $\Ch_B(x)=x^4-(a+b-c)x^2-ac$ and
     $$\sgn(\Ch_B(x))=(+)-(*)(+)+(-),~\text{for}~x\in \mathbb{R}\setminus \{0\},$$
 where $*$ denotes an arbitrary sign.
     Since $V_+(x)=1$ and $V_-(x)=1$, by Descartes' rule of signs, the number of positive, negative real roots of $\Ch_B(x)$ is exactly one. So, $S_B=(2,2)$ and by Lemma~\ref{lem1}, $\In(B)=(1,1,2)$ for all $B\in \mathcal{Q}(\mathcal{P}_3)$. Therefore, $\mathcal{P}_3$ is consistent and requires a unique inertia.

	
	\item
	Since ${\mathcal{P}_{4}}_-=\mathcal{P}_3$, ${\mathcal{P}_{4}}_-$ is consistent and requires a unique inertia.
	Therefore, by Remark \ref{th2.3}, $\mathcal{P}_4$ is consistent and also requires a unique inertia.
    
\end{itemize}

\textbf{Case~4:} $n=5$.\\
Since the number of maximal signed paths of odd length in $\mathcal{P}$ is at most one, the only possible  sign patterns up to equivalence are

$$\mathcal{P}_1=\begin{bmatrix}
	0&+&0&0&0\\+&0&+&0&0\\0&+&0&+&0\\0&0&+&0&+\\0&0&0&+&0
\end{bmatrix},~ \mathcal{P}_2=\begin{bmatrix}
	0&-&0&0&0\\+&0&-&0&0\\0&+&0&-&0\\0&0&+&0&-\\0&0&0&+&0
\end{bmatrix},~ \mathcal{P}_3=\begin{bmatrix}
	0&+&0&0&0\\+&0&+&0&0\\0&+&0&-&0\\0&0&+&0&-\\0&0&0&+&0
\end{bmatrix}.$$

\begin{itemize}
	\item Similarly as in case~1, we have $S_B = (5,0)$ and $\In(B) = (2,2,1)$ for $B \in \mathcal{Q}(\mathcal{P}_1)$; respectively, $S_B = (0,5)$ and $\In(B) = (0,0,5)$ for $B \in \mathcal{Q}(\mathcal{P}_2)$. Therefore, $\mathcal{P}_1, \mathcal{P}_2$ are consistent sign patterns and require a unique inertia.


	\item Every $B\in \mathcal{Q}(\mathcal{P}_3)$ is similar to
	
	$$B=\begin{bmatrix}
		0&1&0&0&0\\ a&0&1&0&0\\ 0&b&0&-1&0\\ 0&0&c&0&-1\\ 0&0&0&d&0
	\end{bmatrix},~ \text{where}~a,b,c,d>0.$$
Therefore, the characteristic polynomial of $B$ is $\Ch_B(x)=x(x^4-(a+b-c-d)x^2-(ac+ad+bd))$. Take, $g(x)=x^4-(a+b-c-d)x^2-(ac+ad+bd)$, then
     $$\sgn(g(x))=(+)-(*)(+)+(-),~\text{for}~x\in \mathbb{R}\setminus \{0\},$$
 where $*$ denotes an arbitrary sign.
     Since $V_+(x)=1$ and $V_-(x)=1$, by Descartes' rule of signs, the number of positive, negative real roots of $g(x)$ is exactly one. So,  $S_B=(3,2)$ and by Lemma~\ref{lem1}, $\In(B)=(1,1,3)$ for all $B\in \mathcal{Q}(\mathcal{P}_3)$. Therefore, $\mathcal{P}_3$ is consistent and also requires a unique inertia.
    
\end{itemize}
\end{proof}




The following result is from Eschenbach et al. \cite{1}.
\begin{lemma}[Lemma 1.5, \cite{1}]\label{l1n}
If an $n\times n$ sign pattern matrix $\mathcal{A}$ does not allow repeated real eigenvalues, then $\mathcal{A}$ is consistent.
\end{lemma}
The converse is not true in general, even for symmetric tree sign patterns. The following example illustrates this fact.
\begin{eg}[Example 1.6, \cite{1}]\rm 
Consider,
$$A=\begin{bmatrix}
	0&+&+&+\\+&0&0&0\\+&0&0&0\\+&0&0&0
\end{bmatrix}$$
Then $A$ is consistent with $S_P=(4,0)$; however, $A$ allows repeated zero eigenvalues.  
\end{eg}

The converse of Lemma \ref{l1n}, whether it holds for tridiagonal sign patterns with a $0$-diagonal, remains an interesting open problem. However, if the order of such patterns is less than or equal to $5$, then we have the following result.

\begin{theorem}\label{nth3.3}
If $\mathcal{P}\in \mathcal{Q}_n$ is an irreducible, tridiagonal sign pattern with a $0$-diagonal, where $n\leq 5$, then $\mathcal{P}$ is consistent if and only if $\mathcal{P}$ does not allow repeated real eigenvalues.
\end{theorem}
\begin{proof}
If $\mathcal{P}$ does not allow repeated real eigenvalues, then by Lemma \ref{l1n}, $\mathcal{P}$ is consistent. To prove the converse, suppose that $\mathcal{P}$ is consistent.

\textbf{Case~1:} $n=2$.\\
The only irreducible, tridiagonal sign patterns with a $0$-diagonal up to equivalence are
$$\mathcal{P}_1=\begin{bmatrix}
	0&+\\+&0
\end{bmatrix},~\mathcal{P}_2=\begin{bmatrix}
	0&-\\+&0
\end{bmatrix}.$$ Clearly $\mathcal{P}_1,\mathcal{P}_2$ are consistent and any $B \in \mathcal{Q}(\mathcal{P}_1)$ (or, $\mathcal{Q}(\mathcal{P}_2)$) has all distinct eigenvalues. 

\textbf{Case~2:} $n=3$. \\
From the proof of Theorem \ref{th3.1} it follows that the only consistent irreducible, tridiagonal sign patterns with a $0$-diagonal up to equivalence are $$\mathcal{P}_1=\begin{bmatrix}
	0&+&0\\+&0&+\\0&+&0
\end{bmatrix}, ~\mathcal{P}_2= \begin{bmatrix}
	0&-&0\\+&0&-\\0&+&0
\end{bmatrix}.$$
Since every $B\in \mathcal{Q}(\mathcal{P}_1)$ (respectively, $B\in \mathcal{Q}(\mathcal{P}_2)$) is similar to a real symmetric (respectively, skew-symmetric) matrix, it is diagonalizable. Hence, the algebraic multiplicity of any eigenvalue $\lambda$ of $B$ is equal to $\rank(B-\lambda I)=1$. So, $B$ has all distinct eigenvalues. 

\textbf{Case~3:} $n=4$.\\
From the proof of Theorem \ref{th3.1} it follows that the only consistent irreducible, tridiagonal sign patterns with a $0$-diagonal up to equivalence are $$\mathcal{P}_1=\begin{bmatrix}
	0&+&0&0\\+&0&+&0\\0&+&0&+\\0&0&+&0
\end{bmatrix}, ~\mathcal{P}_2=\begin{bmatrix}
	0&-&0&0\\+&0&-&0\\0&+&0&-\\0&0&+&0
\end{bmatrix}, ~\mathcal{P}_3=\begin{bmatrix}
	0&+&0&0\\+&0&+&0\\0&+&0&-\\0&0&+&0
\end{bmatrix},~\mathcal{P}_4=\begin{bmatrix}
	0&-&0&0\\+&0&-&0\\0&+&0&+\\0&0&+&0
\end{bmatrix}.$$ 
Every $B \in \mathcal{Q}(\mathcal{P}_1)$ (or, $\mathcal{Q}(\mathcal{P}_2)$) has all distinct eigenvalues, as discussed in case~2. From the proof of Theorem~\ref{th3.1}, it follows that any $B \in \mathcal{Q}(\mathcal{P}_3)$ also has all distinct eigenvalues. Since ${\mathcal{P}_4} = {\mathcal{P}_3}_-$, therefore the eigenvalues of any $B \in \mathcal{Q}(\mathcal{P}_4)$ are all distinct.

\textbf{Case~4:} $n=5$.\\
From the proof of Theorem \ref{th3.1} it follows that the only consistent irreducible, tridiagonal sign patterns with a $0$-diagonal up to equivalence are $$\mathcal{P}_1=\begin{bmatrix}
	0&+&0&0&0\\+&0&+&0&0\\0&+&0&+&0\\0&0&+&0&+\\0&0&0&+&0
\end{bmatrix},~ \mathcal{P}_2=\begin{bmatrix}
	0&-&0&0&0\\+&0&-&0&0\\0&+&0&-&0\\0&0&+&0&-\\0&0&0&+&0
\end{bmatrix},~ \mathcal{P}_3=\begin{bmatrix}
	0&+&0&0&0\\+&0&+&0&0\\0&+&0&-&0\\0&0&+&0&-\\0&0&0&+&0
\end{bmatrix}.$$
Every $B \in \mathcal{Q}(\mathcal{P}_1)$ (or, $\mathcal{Q}(\mathcal{P}_2)$) has all distinct eigenvalues, as discussed in case~2. From the proof of Theorem~\ref{th3.1}, it follows that any $B \in \mathcal{Q}(\mathcal{P}_3)$ also has all distinct eigenvalues.
\end{proof}



The next few results are useful in identifying certain tridiagonal sign patterns with a $0$-diagonal that are not consistent.

\begin{lemma}\label{xnl3}
Let $\mathcal{P} \in \mathcal{Q}_n$ be a sign pattern, whose underlying signed directed graph is $D$. Suppose that the maximum length of composite cycles in $D$ is $m$, with $m$ even. If $D$ contains composite cycles, $\Gamma_1,\Gamma_2$, each of length $m$ and consisting only of $2$-cycles, such that $\Gamma_1,\Gamma_2$ have different number of negative $2$-cycles, then $\mathcal{P}$ is not consistent.
\end{lemma}
\begin{proof}
Let $\Gamma_1 = \alpha_{i_1} \alpha_{i_2} \cdots \alpha_{i_{\frac{m}{2}}}$ and $\Gamma_2 = \alpha_{j_1} \alpha_{j_2} \cdots \alpha_{j_{\frac{m}{2}}}$ be two composite cycles in $D$ consisting only of $2$-cycles of length $m$, where $\alpha_{i_1} \alpha_{i_2},..., \alpha_{i_{\frac{m}{2}}}, \alpha_{j_1}, \alpha_{j_2}, ...,\alpha_{j_{\frac{m}{2}}}$ are $2$-cycles. Let $\Gamma_1$ (respectively, $\Gamma_2$) contain $k_1$ (respectively, $k_2$) numbers of negative $2$-cycles, with $k_1\neq k_2$.

Define the real matrix $B_{\Gamma_1}(0)=[b_{\Gamma_1}(0)_{ij}]$ as in \eqref{xe3}.
Then the eigenvalues of $B_{\Gamma_1}(0)$ are the second complex roots of $10^{2p}$ or $-10^{2p}$, depending on the sign of $\alpha_p$, for $p = 1,2,\dots,\frac{m}{2}$. Hence, $B_{\Gamma_1}(0)$ has $(n - m) + 2k_1$ real eigenvalues.
For $\epsilon>0$, define $B_{\Gamma_1}(\epsilon)\in \mathcal{Q}(\mathcal{P})$ as in \eqref{e2} with $\gamma$ replaced by $\Gamma_1$. Then for $\epsilon > 0$ sufficiently small, eigenvalues of $B_{\Gamma_1}(\epsilon)$ remain close to the eigenvalues of $B_{\Gamma_1}(0)$, therefore all the nonzero real eigenvalues of $B_{\Gamma_1}(\epsilon)$ are algebraically simple. In addition, the real eigenvalues of $B_{\Gamma_1}(0)$ of the form $\lambda$, $-\lambda$ cannot merge to form a nonreal conjugate pair in $B_{\Gamma_1}(\epsilon)$. Also, the nonreal eigenvalues of $B_{\Gamma_1}(0)$ cannot give real eigenvalues of $B_{\Gamma_1}(\epsilon)$ for $\epsilon>0$ sufficiently small. Therefore, $S_{B_{\Gamma_1}(\epsilon)}=S_{B_{\Gamma_1}(0)}=(n-m+2k_1,m-2k_1)$.

Similarly, for $\Gamma_2$, there exists $\epsilon'>0$ such that $B_{\Gamma_2}(\epsilon')\in \mathcal{Q}(\mathcal{P})$ and $S_{B_{\Gamma_2}(\epsilon')}=S_{B_{\Gamma_2}(0)}=(n-m+2k_2,m-2k_2)$. Since $k_1\neq k_2$, therefore $\mathcal{P}$ is not consistent.
\end{proof}

    \begin{eg}\rm 
	Suppose that $\mathcal{P}$ is an irreducible, tridiagonal sign pattern with a $0$-diagonal, whose underlying signed directed graph $D$ is given in Fig.\ref{x1.1}.
		$$\begin{minipage}{.5\textwidth}
	$\mathcal{P}=\begin{bmatrix}
		0&+&0&0&0\\ +&0&+&0&0\\ 0&+&0&-&0\\ 0&0&+&0&+\\0&0&0&+&0
	\end{bmatrix}$
\end{minipage}
\hspace{-1cm}
\begin{minipage}{.45\textwidth}
	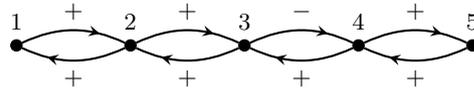
\begin{figure}[H] \label{fign1}
		\tikzset{node distance=1cm}
		\centering
		\usetikzlibrary{decorations.markings}

\tikzset{
  vertex/.style={
    draw, circle, fill=black,
    inner sep=0pt, minimum size=0.15cm
  },
  edge/.style={
    thick,
    postaction={
      decorate,
      decoration={
        markings,
        mark=at position 0.75 with {\arrow{stealth}}
      }
    }
  }
}

\begin{tikzpicture}[scale=1, every node/.style={font=\small}]

\node[vertex,label=above:1] (v1) at (-3,0) {};
\node[vertex,label=above:2] (v2) at (-1.5,0) {};
\node[vertex,label=above:3] (v3) at (0,0) {};
\node[vertex,label=above:4] (v4) at (1.5,0) {};
\node[vertex,label=above:5] (v5) at (3,0) {};

\draw[edge, bend left=25] (v1) to node[above] {$+$} (v2);
\draw[edge, bend left=25] (v2) to node[below] {$+$} (v1);

\draw[edge, bend left=25] (v2) to node[above] {$+$} (v3);
\draw[edge, bend left=25] (v3) to node[below] {$+$} (v2);

\draw[edge, bend left=25] (v3) to node[above] {$-$} (v4);
\draw[edge, bend left=25]  (v4) to node[below] {$+$} (v3);

\draw[edge, bend left=25] (v4) to node[above] {$+$} (v5);
\draw[edge, bend left=25] (v5) to node[below] {$+$} (v4);

\end{tikzpicture}
\caption{The signed directed graph $D$ of $\mathcal{P}$.}

        \label{x1.1}
        \end{figure}
		\end{minipage}   $$   
	Then the maximum length of composite cycles in $D$ is $4$, and $D$ contain two composite cycles, $(1,2)(3,4)$ and $(2,3)(4,5)$, both of length $4$, with the number of negative $2$-cycles equal to $1$ and $2$, respectively. Therefore, $\mathcal{P}$ satisfies the conditions of Lemma~\ref{xnl3}, and hence $\mathcal{P}$ is not consistent.
\end{eg}

Although for tridiagonal sign patterns with a $0$-diagonal, whether consistency requires such patterns to have distinct eigenvalues is not known, however, the following result shows that the multiplicity of zero as an eigenvalue can be at most one.

\begin{theorem} \label{l3.1n}
Let $\mathcal{P}\in \mathcal{Q}_n$ be an irreducible, tridiagonal sign pattern with a $0$-diagonal. Then $\mathcal{P}$ is consistent only if for any $A\in \mathcal{Q}(\mathcal{P})$ has at most, one zero eigenvalue.


\end{theorem}
\begin{proof} Suppose that there exists $A\in \mathcal{Q}(\mathcal{P})$ with more than one zero eigenvalues. Therefore, it follows that  $n$ must be odd.
So, the maximum length of a composite cycle in the signed directed graph $D$ of $\mathcal{P}$ is $n-1$, which is even, consisting only of disjoint $2$-cycles. Moreover, since $\mathcal{P}$ allows repeated zero eigenvalues, the underlying signed directed graph $D$ contains two composite cycles of length $n-1$ with opposite signs. Consequently, $D$ has two composite cycles of length $n-1$, each consisting only of $2$-cycles but with different numbers of negative $2$-cycles. Therefore, by Lemma~\ref{xnl3}, $\mathcal{P}$ is not consistent.
\end{proof}

The conclusion of Theorem \ref{l3.1n} is not necessarily true for tree sign patterns with a $0$-diagonal.

\begin{eg}\rm
Let $$\mathcal{P}=\begin{bmatrix}
	0&+&+&+&+\\+&0&0&0&0\\+&0&0&0&0\\+&0&0&0&0\\+&0&0&0&0
\end{bmatrix}.$$ Then $\rank(B)=2$ for all $B\in \mathcal{Q}(\mathcal{P})$, so $\mathcal{P}$ allows repeated zero eigenvalues. However, since $\mathcal{P}$ is a symmetric tree sign pattern with a $0$-diagonal, by Theorem 1.6 \cite{1991}, it follows that $\mathcal{P}$ is consistent.
\end{eg}

Our next result identifies certain patterns that cannot be submatrices of a consistent irreducible, tridiagonal sign pattern with a $0$-diagonal.
For irreducible path sign patterns with a $0$-diagonal that require a unique inertia, we have the following result from \cite{2025}.

\begin{theorem}[Theorem~2.9, \cite{2025}]\label{nth3.4}
Let $\mathcal{P}\in \mathcal{Q}_n$ be an irreducible, tridiagonal sign pattern with a $0$-diagonal. If $\mathcal{P}_4$ or ${\mathcal{P}_4}_-$ is a submatrix of $\mathcal{P}$, then $\mathcal{P}$ does not require a unique inertia, where
$$\mathcal{P}_4=\begin{bmatrix}
	0&+&0&0\\+&0&-&0\\0&+&0&+\\0&0&+&0
\end{bmatrix}.$$
\end{theorem}

The next result shows that an irreducible, tridiagonal sign pattern with a $0$-diagonal, if it is consistent, cannot have $\mathcal{P}_4, \mathcal{P}_{4_-}$ as a submatrix.

\begin{theorem} \label{th3.51n}
Let $\mathcal{P}\in \mathcal{Q}_n$ be an irreducible, tridiagonal sign pattern with a $0$-diagonal. If $\mathcal{P}_4$ or ${\mathcal{P}_4}_-$ is a submatrix of $\mathcal{P}$, then $\mathcal{P}$ is not consistent, where
 $$\mathcal{P}_4=\begin{bmatrix}
	0&+&0&0\\+&0&-&0\\0&+&0&+\\0&0&+&0
\end{bmatrix}.$$
\end{theorem}
\begin{proof}
The sign pattern $\mathcal{P}$ contains $\mathcal{P}_4$ as a submatrix if and only if $\mathcal{P}_-$ contains $\mathcal{P}_{4-}$ as a submatrix, and by Remark \ref{th2.3}, $\mathcal{P}$ is consistent if and only if $\mathcal{P}_-$ requires a unique inertia. Therefore, the result follows from Theorem \ref{nth3.4}.
\end{proof}

For sign patterns $\mathcal{P}$ of order greater than or equal to $6$, we have the following result from \cite{2025}.

\begin{theorem}[Theorem~2.10, \cite{2025}]\label{nth3.4.4}
Let $\mathcal{P}\in \mathcal{Q}_n$ be an irreducible, tridiagonal sign pattern with a $0$-diagonal. If $\mathcal{P}_6$ or ${\mathcal{P}_6}_-$ is a submatrix of $\mathcal{P}$, then $\mathcal{P}$ does not require a unique inertia, where
$$\mathcal{P}_6=\begin{bmatrix}
	0&+&0&0&0&0\\+&0&-&0&0&0\\0&+&0&-&0&0\\0&0&+&0&-&0\\0&0&0&+&0&+\\0&0&0&0&+&0
\end{bmatrix}.$$
\end{theorem}

A similar result holds for the consistency of an irreducible, tridiagonal sign pattern with a $0$-diagonal.

\begin{theorem} \label{th3.5n}
Let $\mathcal{P}\in \mathcal{Q}_n$ be an irreducible, tridiagonal sign pattern with a $0$-diagonal. If $\mathcal{P}_6$ or ${\mathcal{P}_6}_-$ is a submatrix of $\mathcal{P}$, then $\mathcal{P}$ is not consistent, where 
$$\mathcal{P}_6=\begin{bmatrix}
	0&+&0&0&0&0\\+&0&-&0&0&0\\0&+&0&-&0&0\\0&0&+&0&-&0\\0&0&0&+&0&+\\0&0&0&0&+&0
\end{bmatrix}.$$
\end{theorem}
\begin{proof}
The sign pattern $\mathcal{P}$ contains the submatrix $\mathcal{P}_6$ if and only if $\mathcal{P}_-$ contains ${\mathcal{P}_6}_-$ and by Remark \ref{th2.3}, $\mathcal{P}$ is consistent if and only if $\mathcal{P}_-$ requires a unique inertia. Therefore, the result follows from Theorem \ref{nth3.4.4}.
\end{proof}

Since a sign pattern $\mathcal{P}$ containing $\mathcal{P}_{6}$, ${\mathcal{P}_6}_-$ as a submatrix may not contain $\mathcal{P}_4$, ${\mathcal{P}_{4}}_-$ as a submatrix and vise versa, even patterns which do not contain $\mathcal{P}_4, \mathcal{P}_{4_-}$ or $\mathcal{P}_6, \mathcal{P}_{6_-}$ may well be not consistent.

The following gives a necessary condition for consistency of sign patterns with multiple adjacent leaves in the underlying graph.

\begin{theorem}\label{thx1.1}
Let $\mathcal{P}\in \mathcal{Q}_n$ be an irreducible, combinatorially symmetric sign pattern with a $0$-diagonal, and let $G$ be the underlying signed undirected graph of $\mathcal{P}$. 
If $G$ contains two leaves $u$ and $v$ with a common neighbour $w$, then $\mathcal{P}$ is consistent only if the edges $\{u,w\}$, $\{v,w\}$ have the same sign.
\end{theorem}
\begin{proof} Assume that the edges $\{u,w\}$ and $\{v,w\}$ are oppositely signed, and without loss of generality let $\{u,w\}$ be positive and $\{v,w\}$ be negative. 
Let $D$ denote the signed directed graph corresponding to $\mathcal{P}$, and let $\Gamma$ be a composite cycle of maximum length in $D$.

    \textbf{Case~1:} $(z,w)$ is in $\Gamma$ for some $z\in V(D)$. \\ 
Consider composite cycles $\Gamma_1, \Gamma_2$ in $D$ obtained from $\Gamma$ by replacing $(z,w)$ with $(u,w), (v,w)$, respectively. 
Then both $\Gamma_1, \Gamma_2$ are composite cycles in $D$ of maximum length. 
Define the real matrices $B_{\Gamma_1}(0)$ and $B_{\Gamma_2}(0)$ as in~\eqref{xe3}. 
Then $B_{\Gamma_1}(0)$ has two additional real eigenvalues compared to $B_{\Gamma_2}(0)$.  

For $\epsilon>0$, define $B_{\Gamma_1}(\epsilon), B_{\Gamma_2}(\epsilon) \in \mathcal{Q}(\mathcal{P})$ as in~\eqref{e2}, with $\gamma$ replaced by $\Gamma_1$ and $\Gamma_2$, respectively. Then for $\epsilon>0$ sufficiently small, the eigenvalues of $B_{\Gamma_r}(\epsilon)$ remain close to the eigenvalues of $B_{\Gamma_r}(0)$, for $r=1,2$. 
Therefore, 
$
S_{B_{\Gamma_1}(\epsilon)} \neq S_{B_{\Gamma_2}(\epsilon)}, 
$ and $\mathcal{P}$ is not consistent.

     \textbf{Case~2:} $(z,w)\notin \Gamma$ for any $z\in V(D)$. 
    \\
Then there exists a simple cycle $\gamma=(w,u_1,u_2,...,u_k)$, such that $\gamma \in \Gamma$.

\begin{figure}[H]
\tikzset{node distance=1cm}
\centering
\begin{tikzpicture}
	\tikzstyle{vertex}=[draw, circle, inner sep=0pt, minimum size=.15cm, fill=black]
	\tikzstyle{edge}=[,thick]
	\node[vertex, label=left:$~w$](v1)at(0,0){};
    \node[vertex,label=right:$u$](v2)at(1,0.5){};
	\node[vertex,label=right:$v$](v3)at(1,-0.5){};
    \node[vertex](v8)at(0.35,1){};
	\node[vertex](v9)at(0.35,-1){};
\node[vertex, label=above:$u_1$](v4)at(-.75,1){};
\node[vertex, label=above:$u_2$](v5)at(-1.75,1){};
\node[vertex, label=below:$u_{k-1}$](v6)at(-1.75,-1){};
\node[vertex, label=below:$u_k$](v7)at(-.75,-1){};
    
	\draw[edge](v1)--node[above]{$ +$}(v2);
    \draw[edge](v1)--node[below]{$-$}(v3);
    \draw[edge](v1)--node[]{$ $}(v4);
    \draw[edge](v4)--node[]{$ $}(v5);
    \draw[edge](v1)--node[]{$ $}(v7);
    \draw[edge](v6)--node[]{$ $}(v7);
    \draw[thick, dotted](v5)--node[]{$ $}(v6);
    \draw[thick, dotted](v1)--node[]{$ $}(v8);
    \draw[thick, dotted](v1)--node[]{$ $}(v9);
	\end{tikzpicture}
\caption{$G$}	
\end{figure}
 If $k$ is even then the composite cycle $\Gamma'$ obtained from $\Gamma$ by replacing $\gamma$ with $(u_1,u_2)(u_3,u_4)\cdots (u_{k-1},u_k)\\(v,w)$ has length $l(\Gamma)+1$, which contradicts that $\Gamma$ has maximum length.
Hence $k$ must be odd. 

Consider the composite cycles $\Gamma_1,\Gamma_2$ in $D$ obtained from $\Gamma$ by replacing $\gamma$ with $(u_1,u_2)(u_3,u_4)\cdots (u_{k-2},\\u_{k-1})(u,w)$ and $(u_1,u_2)(u_3,u_4)\cdots (u_{k-2},u_{k-1})(v,w)$, respectively. Then both the composite cycles $\Gamma_1,\Gamma_2$ in $D$ have maximum length. Similarly as in case~1, there exists two real matrices $B_{\Gamma_1}(\epsilon),B_{\Gamma_2}(\epsilon) \in \mathcal{Q}(\mathcal{P})$ with $
S_{B_{\Gamma_1}(\epsilon)} \neq S_{B_{\Gamma_2}(\epsilon)}.
$
Therefore, $\mathcal{P}$ is not consistent.
\end{proof}
We illustrate the above result by the following example.
  \begin{eg} \rm
	Suppose that $\mathcal{P}$ is an irreducible, tridiagonal sign pattern with a $0$-diagonal, whose underlying signed undirected graph $G$ is given in Fig.\ref{x1.2}.
		$$\begin{minipage}{.5\textwidth}
	$\mathcal{P}=\begin{bmatrix}
		0&+&+&+&+\\ +&0&0&0&0\\ -&0&0&0&0\\ +&0&0&0&0\\+&0&0&0&0
	\end{bmatrix}$
\end{minipage}
\hspace{-1cm}
\begin{minipage}{.45\textwidth}
	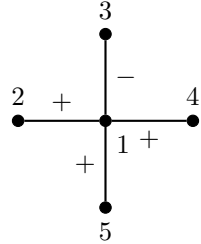
\begin{figure}[H] \label{x1.2}\label{fign1}
		\tikzset{node distance=1cm}
		\centering
		\begin{tikzpicture}
			\tikzstyle{vertex}=[draw, circle, inner sep=0pt, minimum size=.15cm, fill=black]
			\tikzstyle{edge}=[,thick]
			\node[vertex,label=below:~~~~1](v1)at(0,0){};
			\node[vertex,label=above:2](v2)at(-1.15,0){};
			\node[vertex,label=above:4](v3)at(1.15,0){};
			\node[vertex,label=above:3](v4)at(0,1.15){};
			\node[vertex,label=below:5](v5)at(0,-1.15){};
			
	\draw[edge](v1)--node[above]{$ +$}(v2);
    \draw[edge](v1)--node[below]{$ +$}(v3);
    \draw[edge](v1)--node[right]{$ -$}(v4);
    \draw[edge](v1)--node[left]{$ +$}(v5);

		\end{tikzpicture}\caption{The signed undirected graph $ G$ of $\mathcal{P}$.} \end{figure}
		\end{minipage}   $$   
	Then the edges $\{1,2\}$ and $\{1,3\}$ in $G$ are oppositely signed. Therefore, $\mathcal{P}$ satisfies the conditions of Theorem~\ref{thx1.1}, and hence $\mathcal{P}$ is not consistent.
\end{eg}

\section{Consistency of combinatorially symmetric sign patterns with cycles in their underlying signed undirected graph}\label{s4}

In this section, we consider irreducible, combinatorially symmetric sign patterns that necessarily contain cycles but no loops. We derive some interesting necessary and sufficient conditions for such sign patterns to be consistent based on the sign of the edges and other combinatorial properties of the cycles in the underlying graphs.
We begin our discussion with the following lemma.

\begin{lemma}\label{x1}
	Let $\mathcal{P}\in \mathcal{Q}_n$ be a sign pattern, whose underlying directed graph is $D$. Suppose that the maximum length of the composite cycles in $D$ is $m$, $m\geq 2$. If $D$ contains composite cycles $\Gamma_1,\Gamma_2$, consisting only of negative $2$-cycles, positive $2$-cycles, respectively, such that $l(\Gamma_1)+l(\Gamma_2)\geq m+2$, then $\mathcal{P}$ is not consistent.
\end{lemma}
\begin{proof}
	Let $\Gamma_1=\alpha_{i_1}\alpha_{i_2}\cdots \alpha_{i_{k_1}}$, $\Gamma_2=\alpha_{j_1}\alpha_{j_2}\cdots \alpha_{j_{k_2}}$ be two composite cycles in $D$ consisting only of negative $2$-cycles, positive $2$-cycles, respectively, such that $l(\Gamma_1)+l(\Gamma_2)=2k_1+2k_2\geq m+2$.
	
	 Define the real matrix $B_{\Gamma_1}(0)$ as in $\eqref{xe3}$, with $\Gamma$ replaced by $\Gamma_1$.
	Then the eigenvalues of $B_{\Gamma_1}(0)$ are the second complex roots of $10^{2p}$ for $p=1,2,...,{k_1}$. Therefore, 
	$B_{\Gamma_1}(0)$ has $2k_1$ algebraically simple nonzero real eigenvalues. For $\epsilon > 0$ sufficiently small, define $B_{\Gamma_1}(\epsilon)$ as in \eqref{e2} with $\gamma$ replaced by $\Gamma_1$. Then $B_{\Gamma_1}(\epsilon)\in \mathcal{Q}(\mathcal{P})$ and its eigenvalues remain close to the eigenvalues of $B_{\Gamma_1}(0)$. The real eigenvalues of $B_{\Gamma_1}(0)$ of the form $\lambda,-\lambda$ cannot merge to form nonreal conjugate eigenvalue pairs in $B_{\Gamma_1}(\epsilon)$. Hence, 
 $$i_r(B_{\Gamma_1}(\epsilon))\geq (n-m)+2k_1.$$
	
	Similarly, define the real matrix $B_{\Gamma_2}(0)$ as in \eqref{xe3}, with $\Gamma$ replaced by $\Gamma_2$, then it has $2k_2$ algebraically simple, nonzero purely imaginary eigenvalues. Therefore, for $\epsilon>0$ sufficiently small, define $B_{\Gamma_2}(\epsilon)$ as in \eqref{e2} with $\gamma$ replaced by $\Gamma_2$. Then $B_{\Gamma_2}(\epsilon)\in \mathcal{Q}(\mathcal{P})$ and its eigenvalues remain close to the eigenvalues of $B_{\Gamma_2}(0)$. So,
$$i_r(B_{\Gamma_2}(\epsilon))\leq n-2k_2\leq (n-m)+2k_1-2,$$ 
since $2k_1+2k_2\geq m+2$.
    Therefore, $S_{B_{\Gamma_1}(\epsilon)} \neq S_{B_{\Gamma_2}(\epsilon)}$, and $\mathcal{P}$ is not consistent.
\end{proof}

\begin{theorem}\label{x2}
	Let $\mathcal{P} \in \mathcal{Q}_n$ be an irreducible, combinatorially symmetric sign pattern with a $0$-diagonal, whose underlying signed undirected graph $G$ is a cycle. Then $\mathcal{P}$ is not consistent if any one of the following holds.
	\begin{itemize}
		 
		\item[i.] $G$ has two nonadjacent positively signed edges.
		\item[ii.] $G$ has more than one maximal signed path of odd length.
        \item[iii.] $n$ is odd and $G$ has a positively signed edge.
	\end{itemize}
\end{theorem}
\begin{proof}
	Let $D$ be the underlying signed directed graph of $\mathcal{P}$.
	\begin{itemize}


		\item[i.] Suppose that $\{p,p+1\}$ and $\{q,q+1\}$ are two nonadjacent positively signed edges of $G$. 
Consider the composite cycle $\Gamma_1=(p,p+1)(q,q+1)$ in $D$. 
Define the real matrix $B_{\Gamma_1}(0)=[b_{\Gamma_1}(0)_{ij}]$ as in~\eqref{xe3}, with $\Gamma$ replaced by $\Gamma_1$. 
Then the nonzero eigenvalues of $B_{\Gamma_1}(0)$ are $\pm 10$ and $\pm 10^2$. 
Hence, for $\epsilon>0$ sufficiently small, the perturbed matrix $B_{\Gamma_1}(\epsilon)$ has four nonzero real eigenvalues close to those of $B_{\Gamma_1}(0)$.

Now consider the simple cycle $\Gamma_2 = (1,2,\ldots,n)$ in $D$, and define the matrices $B_{\Gamma_2}(0)$ and $B_{\Gamma_2}(\epsilon)$ as in~\eqref{e1} and~\eqref{e2}, with $\gamma$ replaced by $\Gamma_2$.
Since $B_{\Gamma_2}(0)$ has $n$ nonzero algebraically simple eigenvalues, namely the $n$-th complex roots of unity (for $n\geq 3$), it follows that for $\epsilon>0$ sufficiently small, the perturbed matrix $B_{\Gamma_2}(\epsilon)\in\mathcal{Q}(\mathcal{P})$ and has $n$ nonzero algebraically simple eigenvalues close to those of $B_{\Gamma_2}(0)$. Hence $B_{\Gamma_2}(\epsilon)$ has at most $2$ of which is real. 

Therefore,
$
S_{B_{\Gamma_1}(\epsilon)}\neq S_{B_{\Gamma_2}(\epsilon)}
$, and hence $\mathcal{P}$ is not consistent.

		\item[ii.] Assume $G$ has at least two maximal signed paths of odd length.
		
		\textbf{Case~1:} $n$ is odd. \\
		Then there exist at least three maximal signed paths of odd length. Let $P_1=\{1,2\}\{2,3\} \cdots \{2p-1,2p\}$ and $P_2=\{s,s+1\}\{s+1,s+2\}\cdots \{s+2r,s+2r+1\}$ be two maximal signed paths of odd length, where $s> 2p$. Since $n$ is odd, either $\dist(2p,s)$ is odd or $\dist(s+2r+1,1)$ is odd. Without loss of generality, suppose that $\dist(2p,s)$ is odd, then $\dist(s,1)$ is also odd.
		Consider,
		$$
		\begin{aligned}
			\mathcal{S}={} & \bigl\{\{2i-1,2i\} : 1\le i\le p\bigr\}
			\;\cup\;
			\bigl\{\{2i,2i+1\} :2p \le 2i \le (s-1)\bigr\} \\[6pt]
			& \;\cup\;
			\bigl\{\{2i-1,2i\} : (s+1) \le 2i \le (n-1)\bigr\}
			\;\cup\;
			\bigl\{\{n,1\}\bigr\}.
		\end{aligned}
		$$
		
		\begin{figure}[H]
			\tikzset{node distance=1cm}
			\centering
			\begin{tikzpicture}
				\tikzstyle{vertex}=[draw, circle, inner sep=0pt, minimum size=.15cm, fill=black]
				\tikzstyle{edge}=[,thick]
				\node[vertex,label=right:1](v1)at(2,0){};
				\node[vertex,label=above:9](v2)at(1.53,1.28){};
				\node[vertex,label=above:8](v3)at(0.35,1.97){};
				\node[vertex,label=above:7](v4)at(-1.18,1.62){};
					\node[vertex,label=left:6](v5)at(-1.90,0.62){};
				\node[vertex,label=left:5](v6)at(-1.90,-0.62){};
					\node[vertex,label=below:4](v7)at(-1.18,-1.62){};
				\node[vertex,label=below:3](v8)at(0.35,-1.97){};
				\node[vertex,label=below:2](v9)at(1.53,-1.28){};
				\draw[edge, ultra thick, dotted](v1)--node[right, yshift=0.2cm, xshift=-0.1cm]{$+$}(v2);
				\draw[edge](v2)--node[above,  xshift=0.1cm]{$+$}(v3);
				\draw[edge, ultra thick, dotted](v3)--node[above, yshift=0.0cm, xshift=-0.1cm]{$-$}(v4);
				\draw[edge, ultra thick, dotted](v4)--node[left, yshift=0.1cm, xshift=0cm]{$+$}(v5);
				\draw[edge](v5)--node[left, yshift=0cm, xshift=-0.05cm]{$+$}(v6);
				\draw[edge, ultra thick, dotted](v6)--node[left, yshift=-0.1cm]{$+$}(v7);
				\draw[edge, ultra thick, dotted](v7)--node[below]{$-$}(v8);
				\draw[edge](v8)--node[below]{$-$}(v9);
				\draw[edge, ultra thick, dotted](v9)--node[right, yshift=-0.05cm, xshift=0cm]{$-$}(v1);
			\end{tikzpicture}\caption{ In the above figure $G$ is a cycle of length $9$, $P_1=\{1,2\}\{2,3\}\{3,4\}$ and $P_2=\{7,8\}$ where the dotted edges correspond to $\mathcal{S}$.} \end{figure}
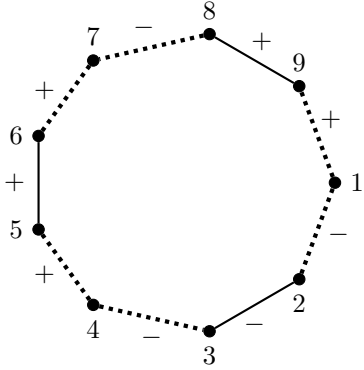
		Then the edges $\{2p-1,2p\}$, $\{2p,2p+1\}$; $\{s-1,s\}$, $\{s,s+1\}$; $\{n,1\}$, $\{1,2\}$ are the only pair of edges in $\mathcal{S}$ which are adjacent, and the edges in each pair are signed oppositely. Hence, the positive edges in $\mathcal{S}$, the negative edges in $\mathcal{S}$ are all nonadjacent. Let $k_1$ (respectively, $k_2$) be the number of positive (respectively, negative) edges in $\mathcal{S}$. Since every vertex of $G$ occurs once in $\mathcal{S}$ except $2p$, $s$ and $1$, which occurs twice, hence $2k_1+2k_2=n+3$. Let $\mathcal{M}_1$ 
        and $\mathcal{M}_2$ 
        be two matchings in $G$ containing only positive and negative edges in $\mathcal{S}$, respectively. 
		
		 Let $\Gamma_1$, $\Gamma_2$ be the composite cycles in $D$ corresponding to $\mathcal{M}_1$, $\mathcal{M}_2$, respectively. Then, $\Gamma_1$ (respectively, $\Gamma_2$) consists only of negative (respectively, positive) $2$-cycles in $D$, and since $l(\Gamma_1) + l(\Gamma_2) = n + 3$, it follows from Lemma~\ref{x1} that $\mathcal{P}$ is not consistent.

		\textbf{Case~2:} $n$ is even.\\
          Suppose $G$ contains more than one maximal signed path of odd length. Let $P_1=\{1,2\}\{2,3\}\cdots \{2p-1,2p\}$ be a maximal signed path of odd length. Consider,
			$$\mathcal{S}=\{\{2j-1,2j\}: 1\leq j \leq p\}~\cup~\{\{2i,2i+1\}: 2p\leq 2i\leq (n-2)\}~\cup~\{\{n,1\}\}. $$
            Then the edges $\{2p-1,2p\}$, $\{2p,2p+1\}$; $\{n,1\}$, $\{1,2\}$ are the only pair of edges in $\mathcal{S}$ which are adjacent, and the edges in each pair are signed oppositely. Then similarly as in case~1, $\mathcal{S}$ contains two matchings, denoted by $\mathcal{M}_1$ and $\mathcal{M}_2$, where $\mathcal{M}_1$ (respectively, $\mathcal{M}_2$) consists only of positively (respectively, negatively) signed edges from $\mathcal{S}$ and has length $k_1$ (respectively, $k_2$), with $2k_1 + 2k_2 = n + 2$. 
			
			 Let $\Gamma_1$, $\Gamma_2$ be the composite cycles in $D$ corresponding to $\mathcal{M}_1$, $\mathcal{M}_2$, respectively. Then, $\Gamma_1$ (respectively, $\Gamma_2$) consists only of negative (respectively, positive) $2$-cycles in $D$, and since $l(\Gamma_1) + l(\Gamma_2) = n + 2$, it follows from Lemma~\ref{x1} that $\mathcal{P}$ is not consistent.
             
             \item[iii.] Suppose that $\{p,p+1\}$ is a positively signed edge in $G$. 
Consider the cycle $\Gamma_1 = (p,p+1)$ in $D$, and define the matrices $B_{\Gamma_1}(0)$ and $B_{\Gamma_1}(\epsilon)$ as in~\eqref{e1} and~\eqref{e2}, with $\gamma$ replaced by $\Gamma_1$. 
For $\epsilon>0$ sufficiently small, the perturbed matrix $B_{\Gamma_1}(\epsilon)$ has two nonzero real eigenvalues close to $\pm 1$. 
Hence $i_r(B_{\Gamma_1}(\epsilon)) \geq 2$.

Now consider the simple cycle $\Gamma_2 = (1,2,\ldots,n)$ in $D$, and define the matrices $B_{\Gamma_2}(0),B_{\Gamma_2}(\epsilon)$ as in~\eqref{e1},~\eqref{e2}, with $\gamma$ replaced by $\Gamma_2$. 
Then for $\epsilon>0$ sufficiently small, the eigenvalues of $B_{\Gamma_2}(\epsilon)$ are algebraically simple and lie close to the eigenvalues of $B_{\Gamma_2}(0)$. Since $n$ is odd, $B_{\Gamma_2}(\epsilon)$ has exactly one real eigenvalue. Thus $i_r(B_{\Gamma_2}(\epsilon))=1$. Therefore, $\mathcal{P}$ is not consistent.
	\end{itemize}
\end{proof}	

The following are examples of sign patterns that satisfy conditions (i), (ii), and (iii), respectively, of Theorem~\ref{x2} and therefore do not require a unique inertia.

\begin{eg}\rm
	 Suppose that $\mathcal{P}\in \mathcal{Q}_4$ is an irreducible, combinatorially symmetric sign pattern with a $0$-diagonal, whose underlying signed undirected graph is the graph $G$ given in Fig~\ref{x3}.
	$$ \begin{minipage}{.5\textwidth}
		$\mathcal{P}=\begin{bmatrix}
			0&+&0&+\\+&0&+&0\\0&+&0&+\\+&0&+&0
		\end{bmatrix}$
	\end{minipage}
	\hspace{-3cm}
	\begin{minipage}{.45\textwidth}
		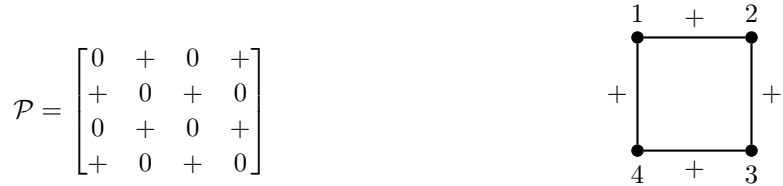
\begin{figure}[H] \label{x3}
			\tikzset{node distance=1cm}
			\centering
			\begin{tikzpicture}
				\tikzstyle{vertex}=[draw, circle, inner sep=0pt, minimum size=.15cm, fill=black]
				\tikzstyle{edge}=[,thick]
				\node[vertex,label=above:1](v1)at(-.75,.75){};
				\node[vertex,label=above:2](v2)at(.75,.75){};
				\node[vertex,label=below:3](v3)at(.75,-.75){};
				\node[vertex,label=below:4](v4)at(-.75,-.75){};

				\draw[edge](v1)--node[above]{$+ $}(v2);
				\draw[edge](v2)--node[right]{$+ $}(v3);
				\draw[edge](v3)--node[below]{$+ $}(v4);
				\draw[edge](v1)--node[left]{$+ $}(v4);

			\end{tikzpicture}\caption{The signed undirected graph $ G$ of $\mathcal{P}$.}\end{figure}
            
	\end{minipage}   $$ 
	Then $G$ is a cycle with two nonadjacent positively signed edges $\{1,2\},\{3,4\}$.  Thus, $\mathcal{P}$ satisfies the conditions stated in (i) of Theorem~\ref{x2}. Consequently, $\mathcal{P}$ is not consistent. 
\end{eg}

\begin{eg}\rm 
	 Suppose that $\mathcal{P}\in \mathcal{Q}_4$ is an irreducible, combinatorially symmetric sign pattern with a $0$-diagonal, whose underlying signed undirected graph is the graph $G$ given in Fig~\ref{x4}.
	$$ \begin{minipage}{.5\textwidth}
		$\mathcal{P}=\begin{bmatrix}
			0&-&0&+\\+&0&-&0\\0&+&0&-\\+&0&+&0
		\end{bmatrix}$
	\end{minipage}
	\hspace{-3cm}
	\begin{minipage}{.45\textwidth}
		\begin{figure}[H] \label{x4}
			\tikzset{node distance=1cm}
			\centering
			\begin{tikzpicture}
				\tikzstyle{vertex}=[draw, circle, inner sep=0pt, minimum size=.15cm, fill=black]
				\tikzstyle{edge}=[,thick]
				\node[vertex,label=above:1](v1)at(-.75,.75){};
				\node[vertex,label=above:2](v2)at(.75,.75){};
				\node[vertex,label=below:3](v3)at(.75,-.75){};
				\node[vertex,label=below:4](v4)at(-.75,-.75){};

				\draw[edge](v1)--node[above]{$- $}(v2);
				\draw[edge](v2)--node[right]{$- $}(v3);
				\draw[edge](v3)--node[below]{$- $}(v4);
				\draw[edge](v1)--node[left]{$+ $}(v4);

			\end{tikzpicture}\caption{The signed undirected graph $ G$ of $\mathcal{P}$.} \end{figure}
	\end{minipage}   $$ 
	Then $G$ is a cycle with more than one maximal signed path of odd length.  Thus, $\mathcal{P}$ satisfies the conditions stated in (ii) of Theorem~\ref{x2}. Consequently, $\mathcal{P}$ is not consistent.
\end{eg}
\begin{eg}\rm 
    Suppose that $\mathcal{P}\in \mathcal{Q}_3$ is an irreducible, combinatorially symmetric sign pattern with a $0$-diagonal, whose underlying signed undirected graph $G$ is given in Fig.\ref{fign1.1}.
	$$ \begin{minipage}{.5\textwidth}
		$\mathcal{P}=
\begin{bmatrix}
	0 & + & + \\
	+ & 0 & + \\
	+ & + & 0
\end{bmatrix}$
	\end{minipage}
	\hspace{-3cm}
	\begin{minipage}{.45\textwidth}
		\begin{figure}[H]
	\centering
	\tikzset{node distance=1cm}
	\begin{tikzpicture}
		\tikzstyle{vertex}=[draw, circle, inner sep=0pt, minimum size=.15cm, fill=black]
		\tikzstyle{edge}=[thick]
		\node[vertex,label=left:3](v1) at(-1,0) {};
		\node[vertex,label=right:2](v2) at(1,0) {};
		\node[vertex,label=above:1](v3) at(0,1) {};
		\draw[edge](v1)--node[below]{$+$}(v2);
		\draw[edge](v2)--node[right]{$+$}(v3);
		\draw[edge](v3)--node[left]{$+$}(v1);
	\end{tikzpicture}
	\caption{The signed undirected graph $ G$ of $\mathcal{P}$.}
	\label{fign1.1}
\end{figure}
	\end{minipage}   $$ 
    Then $G$ is a cycle with a positively signed edge.  Thus, $\mathcal{P}$ satisfies the conditions stated in (iii) of Theorem~\ref{x2}. Consequently, $\mathcal{P}$ is not consistent.
\end{eg}


			

The next example shows that the conditions given in Theorem~\ref{x2} are not sufficient to guarantee the consistency of a sign pattern.
\begin{eg}\label{nex1}\rm 
	 Suppose that $\mathcal{P}\in \mathcal{Q}_4$ is an irreducible, combinatorially symmetric sign pattern with a $0$-diagonal, whose underlying signed undirected graph $G$ is given in Fig.\ref{fign1.12}.
	$$ \begin{minipage}{.5\textwidth}
		$\mathcal{P}=
\begin{bmatrix}
	0 & - & 0 & -\\
	+ & 0 & - & 0 \\
	0 & + & 0 & +\\
    + & 0 & - & 0
\end{bmatrix}$
	\end{minipage}
	\hspace{-3cm}
	\begin{minipage}{.45\textwidth}
		\begin{figure}[H] 
			\tikzset{node distance=1cm}
			\centering
			\begin{tikzpicture}
				\tikzstyle{vertex}=[draw, circle, inner sep=0pt, minimum size=.15cm, fill=black]
				\tikzstyle{edge}=[,thick]
				\node[vertex,label=above:1](v1)at(-.75,.75){};
				\node[vertex,label=above:2](v2)at(.75,.75){};
				\node[vertex,label=below:3](v3)at(.75,-.75){};
				\node[vertex,label=below:4](v4)at(-.75,-.75){};

				\draw[edge](v1)--node[above]{$- $}(v2);
				\draw[edge](v2)--node[right]{$- $}(v3);
				\draw[edge](v3)--node[below]{$- $}(v4);
				\draw[edge](v1)--node[left]{$- $}(v4);

			\end{tikzpicture}\caption{The signed undirected graph $ G$ of $\mathcal{P}$.} \label{fign1.12} \end{figure}
	\end{minipage}   $$    
	Then $G$ is a cycle with exactly one maximal signed path of odd length, and no positively signed edge. Therefore, $\mathcal{P}$ does not satisfy any of the conditions of Theorem~\ref{x2}.
   Take
$$B_1=\begin{bmatrix}
			0&-10&0&-1\\1&0&-10&0\\0&1&0&10\\10&0&-1&0
		\end{bmatrix},~ B_2=\begin{bmatrix}
			0&-10&0&-1\\10&0&-1&0\\0&1&0&10\\1&0&-10&0
		\end{bmatrix}\in \mathcal{Q}(\mathcal{P}).$$
        Then $S_{B_1}=(2,2)$ and $S_{B_2}=(0,4)$. Therefore, $\mathcal{P}$ is not consistent.
\end{eg}

Note that Theorem~\ref{x2} provides necessary conditions for the consistency of an irreducible, combinatorially symmetric sign pattern with a $0$-diagonal, whose undirected graph is a cycle. Moreover, Condition (iii) tells us that $\mathcal{P}$ is consistent and of odd order, then the underlying undirected graph $G$ of $\mathcal{P}$ can not have any positive edge. In continuation, the following result shows that $G$ of odd order is consistent if and only if $G$ does not have any positive edge.

\begin{theorem}
Let $\mathcal{P} \in \mathcal{Q}_n$ be an irreducible, combinatorially symmetric sign pattern of odd order $n$ with $0$-diagonal, whose underlying signed undirected graphs $G$ is a cycle. 
If all edges of $G$ are negatively signed, then $\mathcal{P}$ is consistent.
\end{theorem}
\begin{proof} Suppose that $G$ is a cycle of odd length in which all edges are negatively signed. Then for any $B\in \mathcal{Q}(\mathcal{P})$, $E_k(B) > 0$ if $k$ is even, $E_k(B) = 0$ if $k$ is odd and $k<n$, where $E_k(B)$ for $1\leq k\leq n$ is the sum of all cycles (simple or composite) of length $k$ in $B$, properly signed.
Therefore, the characteristic polynomial of $B$ is 
$$\Ch_B(x)=x^n + E_2(B)x^{n-2} + \cdots + E_{n-1}(B)x - E_n(B).$$ 

\begin{itemize}
\item[i.] If $E_n(B)>0$, then
$$
\sgn(\Ch_B(x)) = 
\begin{cases} 
(+)+(+)(+)+\cdots+(+)(+)-(+) & \text{if $x>0$} \\
(-)+(+)(-)+\cdots+(+)(-)-(+) & \text{if $x<0$}.
\end{cases}
$$
Hence $V_+(x)=1$ and $V_-(x)=0$, so by Descartes' rule of signs, $\Ch_B(x)$ has exactly one positive root and no negative or zero roots.

\item[ii.] If $E_n(B)<0$, then
$$
\sgn(\Ch_B(x)) = 
\begin{cases} 
(+)+(+)(+)+\cdots+(+)(+)-(-) & \text{if $x>0$} \\
(-)+(+)(-)+\cdots+(+)(-)-(-) & \text{if $x<0$}.
\end{cases}
$$
Hence $V_+(x)=0$ and $V_-(x)=1$, so by Descartes' rule of signs, $\Ch_B(x)$ has exactly one negative root and no positive or zero roots.

\item[iii.] If $E_n(B)=0$, then
$\Ch_B(x) \;=\; x \big(x^{\,n-1}+E_2(B)x^{\,n-3}+\cdots+E_{n-1}(B)\big)$. Clearly $0$ is the only real root of $\Ch_B(x)$.
\end{itemize}

Therefore, any $B \in \mathcal{Q}(\mathcal{P})$ has exactly one real eigenvalue. Consequently, $S_\mathcal{P} = (1, n-1)$.
\end{proof}

The following examples show that if the order of $G$ is even and all the edges in $G$ are negative, then we cannot conclude either way, that is, in that case, the associated sign pattern may or may not be consistent.

\begin{eg}\rm
Suppose that $\mathcal{P}\in \mathcal{Q}_4$ is an irreducible combinatorially symmetric sign pattern with a $0$-diagonal, whose underlying signed undirected graph $G$ is given in Fig.\ref{fign1}.
$$ \begin{minipage}{.5\textwidth}
$\mathcal{P}=\begin{bmatrix}
	0&-&0&-\\+&0&-&0\\0&+&0&-\\+&0&+&0
\end{bmatrix}$
\end{minipage}
\hspace{-3cm}
\begin{minipage}{.45\textwidth}
\begin{figure}[H] \label{fign1}
	\tikzset{node distance=1cm}
	\centering
	\begin{tikzpicture}
		\tikzstyle{vertex}=[draw, circle, inner sep=0pt, minimum size=.15cm, fill=black]
		\tikzstyle{edge}=[,thick]
		\node[vertex,label=above:1](v1)at(-.75,.75){};
		\node[vertex,label=above:2](v2)at(.75,.75){};
		\node[vertex,label=below:3](v3)at(.75,-.75){};
		\node[vertex,label=below:4](v4)at(-.75,-.75){};

		\draw[edge](v1)--node[above]{$- $}(v2);
		\draw[edge](v2)--node[right]{$- $}(v3);
		\draw[edge](v3)--node[below]{$-$}(v4);
		\draw[edge](v1)--node[left]{$- $}(v4);

	\end{tikzpicture}\caption{the signed undirected graph $ G$ of $\mathcal{P}$.} \end{figure}
\end{minipage}   $$    
Then $G$ is a cycle with all edges are negatively signed. Suppose that $B\in \mathcal{Q}(\mathcal{P})$, then $B$ is similar to
$$\begin{bmatrix}
0&-a&0&-b\\c&0&-d&0\\0&e&0&-f\\g&0&h&0
\end{bmatrix},~\text{where}~a,b,c,d,e,f,g,h>0.$$
The characteristic polynomial of $B$ is $\Ch_B(x)=x^4+(ac+bg+ed+fh)x^2+(acfh+adfg+ebch+ebdg)$. Therefore,
$$\sgn(\Ch_B(x)) = 
(+)+(+)(+)+(+) ~\text{for all} ~x\in \mathbb{R}\setminus\{0\}. 
$$ Hence $V_+(x)=0$ and $V_-(x)=0$, so by Descartes' rule of signs, the number of positive and negative real roots of $\Ch_B(x)$ is zero. Therefore, $S_B=(0,4)$ for all $B\in \mathcal{Q}(\mathcal{P})$, so $\mathcal{P}$ is consistent.
\end{eg}

\begin{eg}\rm
Let $\mathcal{P}\in \mathcal{Q}_4$ be an irreducible, combinatorially symmetric sign pattern with a $0$-diagonal, whose underlying undirected graph is $G$ as in Example~\ref{nex1}. Then $G$ is a cycle with all edges negatively signed, and by Example~\ref{nex1}, $\mathcal{P}$ is not consistent.
\end{eg}
Recall that, by Condition (i) of Theorem~\ref{x2}, if  $\mathcal{P}\in \mathcal{Q}_n$  is consistent and either $n$ is odd with $n\geq 5$ or $n$ is even with $n\geq 4$, then the underlying signed undirected graph $G$ of $\mathcal{P}$ can have at most two positive edges.
The following examples demonstrate that for even $n$, if the underlying graph has exactly two positive edges, then $\mathcal{P}$ need not be consistent.

\begin{eg}\rm
Suppose that $\mathcal{P}\in \mathcal{Q}_4$ is an irreducible combinatorially symmetric sign pattern with a $0$-diagonal, whose underlying signed undirected graph $G$ is given in Fig.\ref{fign1.2}.
$$ \begin{minipage}{.5\textwidth}
$\mathcal{P}=\begin{bmatrix}
	0&-&0&+\\+&0&-&0\\0&+&0&+\\+&0&+&0
\end{bmatrix}$
\end{minipage}
\hspace{-3cm}
\begin{minipage}{.45\textwidth}
\begin{figure}[H] \label{fign1.2}
	\tikzset{node distance=1cm}
	\centering
	\begin{tikzpicture}
		\tikzstyle{vertex}=[draw, circle, inner sep=0pt, minimum size=.15cm, fill=black]
		\tikzstyle{edge}=[,thick]
		\node[vertex,label=above:1](v1)at(-.75,.75){};
		\node[vertex,label=above:2](v2)at(.75,.75){};
		\node[vertex,label=below:3](v3)at(.75,-.75){};
		\node[vertex,label=below:4](v4)at(-.75,-.75){};

		\draw[edge](v1)--node[above]{$- $}(v2);
		\draw[edge](v2)--node[right]{$- $}(v3);
		\draw[edge](v3)--node[below]{$+$}(v4);
		\draw[edge](v1)--node[left]{$+ $}(v4);

	\end{tikzpicture}\caption{The signed undirected graph $ G$ of $\mathcal{P}$.} \end{figure}
\end{minipage}   $$    
Then $G$ is a cycle with exactly two negative edges. Suppose that $B\in \mathcal{Q}(\mathcal{P})$, then $B$ is similar to
$$\begin{bmatrix}
0&-a&0&b\\c&0&-d&0\\0&e&0&f\\g&0&h&0
\end{bmatrix},~\text{where}~a,b,c,d,e,f,g,h>0.$$
The characteristic polynomial of $B$ is $\Ch_B(x)=x^4+(ac-bg+ed-fh)x^2-(acfh+adfg+ebch+ebdg)$. Therefore,
$$\sgn(\Ch_B(x)) = 
(+)+(*)(+)-(+) ~\text{for all} ~x\in \mathbb{R}\setminus\{0\}, 
$$ where $*$ denotes an arbitrary sign. Hence $V_+(x)=1$ and $V_-(x)=1$, so by Descartes' rule of signs, the number of positive and negative real roots of $\Ch_B(x)$ is exactly one. Therefore, $S_B=(2,2)$ for all $B\in \mathcal{Q}(\mathcal{P})$, so $\mathcal{P}$ is consistent.
\end{eg}

\begin{eg}\rm
	 Suppose that $\mathcal{P}\in \mathcal{Q}_4$ is an irreducible, combinatorially symmetric sign pattern with a $0$-diagonal, whose underlying signed undirected graph $G$ is given in Fig.\ref{fign1.3}.
	$$ \begin{minipage}{.5\textwidth}
		$\mathcal{P}=
\begin{bmatrix}
	0 & + & 0 & +\\
	- & 0 & - & 0 \\
	0 & + & 0 & +\\
    + & 0 & + & 0
\end{bmatrix}$
	\end{minipage}
	\hspace{-3cm}
	\begin{minipage}{.45\textwidth}
		\begin{figure}[H] \label{fign1.3}
			\tikzset{node distance=1cm}
			\centering
			\begin{tikzpicture}
				\tikzstyle{vertex}=[draw, circle, inner sep=0pt, minimum size=.15cm, fill=black]
				\tikzstyle{edge}=[,thick]
				\node[vertex,label=above:1](v1)at(-.75,.75){};
				\node[vertex,label=above:2](v2)at(.75,.75){};
				\node[vertex,label=below:3](v3)at(.75,-.75){};
				\node[vertex,label=below:4](v4)at(-.75,-.75){};

				\draw[edge](v1)--node[above]{$- $}(v2);
				\draw[edge](v2)--node[right]{$- $}(v3);
				\draw[edge](v3)--node[below]{$+ $}(v4);
				\draw[edge](v1)--node[left]{$+ $}(v4);

			\end{tikzpicture}\caption{the signed undirected graph $ G$ of $\mathcal{P}$.} \end{figure}
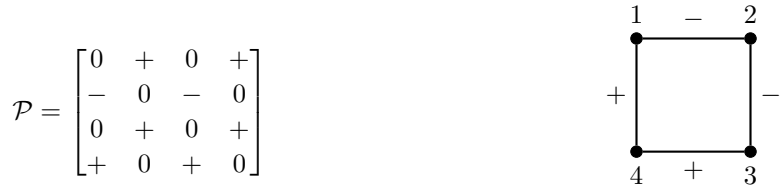
	\end{minipage}   $$    
	Then $G$ is a cycle with exactly two negative edges.
   Take
$$B_1=\begin{bmatrix}
			0&1&0&1\\-1&0&-1&0\\0&1&0&1\\1&0&1&0
		\end{bmatrix},~ B_2=\begin{bmatrix}
			0&2&0&1\\-1&0&-1&0\\0&1&0&1\\1&0&1&0
		\end{bmatrix}\in \mathcal{Q}(\mathcal{P}).$$
        Then $S_{B_1}=(4,0)$ and $S_{B_2}=(2,2)$. Therefore, $\mathcal{P}$ is not consistent.
\end{eg}
Suppose that $\mathcal{P}$ is a combinatorially symmetric sign pattern whose underlying signed undirected graph $G$ is connected and contains a cycle but no loop. We establish necessary conditions for such sign patterns $\mathcal{P}$ to be consistent. For this, we recall the following lemma from~\cite{2025}.

\begin{lemma}[Lemma~3.2, \cite{2025}]\label{l28}
Let $\mathcal{P}\in \mathcal{Q}_n$ be an irreducible, combinatorially symmetric sign pattern with a $0$-diagonal, whose underlying signed undirected graph is $G$ and the signed directed graph is $D$. Suppose that $G$ has exactly one cycle $\mathcal{C}=u_1u_2\cdots u_k$ of length $k$ and the distance of $\mathcal{C}$ from any leaf in $G$ is even. Then the maximum length of a composite cycle in $D$ is equal to the sum of the length of $\mathcal{C}$ and the maximum length of a composite cycle in $D\setminus V(\mathcal{C})$.       
\end{lemma}

\begin{theorem}\label{xx1}
Let $\mathcal{P}\in \mathcal{Q}_n$ be an irreducible, combinatorially symmetric sign pattern with a $0$-diagonal, whose underlying signed undirected graph $G$ has exactly one cycle $\mathcal{C}=u_1u_2\cdots u_k$, and the distance of $\mathcal{C}$ from any leaf in $G$ is even. Then $\mathcal{P}$ is not consistent if any one of the following holds.
\begin{itemize}
\item[i.] $\mathcal{C}$ has at least two nonadjacent positively signed edges.
\item[ii.] $ \mathcal{C}$ has more than one maximal signed path of odd length.
\item[iii.] $k$ is odd and $ \mathcal{C}$ has a positively signed edge. 
\end{itemize}
\end{theorem}
\begin{proof}
   Assume that $G$ does not have leaves, then Theorem~\ref{x2} implies that $\mathcal{P}$ is not consistent. Therefore, suppose that $G$ contains at least one leaf. Let $D$ be the underlying signed directed graph of $\mathcal{P}$ and consider a composite cycle $\Gamma = \beta_1 \beta_2 \cdots \beta_l$
of maximum length in $D \setminus V(\mathcal{C})$, where $l(\Gamma) = 2l$ and each $\beta_r$ is a $2$-cycle for $r=1,2,\dots,l$. 
By Lemma~\ref{l28}, it follows that the maximum length of a composite cycle in $D$ is $m = k + 2l$.
\begin{itemize}
    \item[i.]  Suppose that $\mathcal{C}$ has at least two nonadjacent positively signed edges. Then similarly as in the proof of Theorem~\ref{x2}(i), there exist two composite cycles
$\Gamma_1,\Gamma_2$ in $D$ and two real matrix $B_{\Gamma_1}(\epsilon) = [b_{\Gamma_1}(\epsilon)_{ij}],B_{\Gamma_2}(\epsilon) = [b_{\Gamma_2}(\epsilon)_{ij}] \in \mathcal{Q}(\mathcal{P}[\{u_1,u_2,...,u_k\}])$ with  $S_{B_{\Gamma_1}(\epsilon)}\neq S_{B_{\Gamma_2}(\epsilon)}$. Moreover, $B_{\Gamma_1}(\epsilon)$ has four algebraically simple real eigenvalues close to $\pm 10$ and $\pm 10^2$ and the other eigenvalues lies close to zero, while the eigenvalues of $B_{\Gamma_2}(\epsilon)$ are close to the $k$-th roots of unity.
Let,
$$M>\max\{1,|\lambda|: \lambda~ \text{is an eigenvalue of either} ~B_{\Gamma_1}(\epsilon)~\text{or}~B_{\Gamma_2}(\epsilon)\}.$$ 
Consider $$\Gamma_1'=\Gamma_1\Gamma \quad\text{and}\quad \Gamma_2'=\Gamma_2\Gamma,$$
and define the real matrices $B_{\Gamma_t'}(\epsilon)=[b_{\Gamma_t'}(\epsilon)_{ij}]\in \mathcal{Q}(\mathcal{P})$ for $t=1,2$, as
$$|b_{\Gamma_t'}(\epsilon)_{ij}|= \begin{cases}
	|b_{\Gamma_t}(\epsilon)_{ij}| & \text{if}\ p_{ij}~ \text{is in } \Gamma_t \\
	M^p & \text{if}\ p_{ij}~ \text{is in}~ \beta_p,~\text{for}~p=1,2,...,l \\
	\epsilon & \text{if}~ p_{ij}\neq 0~\text{and not in}~ \Gamma_t, \Gamma\\
	0 & \text{elsewhere.}
\end{cases}$$
For $\epsilon > 0$ sufficiently small, the eigenvalues of $B_{\Gamma_t'}(\epsilon)$ remain close to the eigenvalues of $B_{\Gamma_t}(\epsilon)$  for $t=1,2$, together with the second complex roots of $M^{2p}$ or $-M^{2p}$ depending on the sign of $\beta_p$, for $p=1,2,...,l$.
Therefore, $i_r(B_{\Gamma_1'}(\epsilon))\geq 4+2s+(n-m)$ and $i_r(B_{\Gamma_2'}(\epsilon))\leq 2+2s+(n-m)$, where $s$ be the number of negative $2$-cycles in $\Gamma$.
Hence, $S_{B_{\Gamma_1'}(\epsilon)}\neq S_{B_{\Gamma_2'}(\epsilon)}$,  $\mathcal{P}$ is not consistent.

\item[ii.]  If $\mathcal{C}$ has more than one maximal signed path of odd length, then similarly as in the proof of Theorem \ref{x2}(ii), $G$ has two matchings $\mathcal{M}_1$ consisting only of positive edges and $\mathcal{M}_2$ consisting only of negative edges in $\mathcal{C}$, such that  
$$l(\mathcal{M}_1)+l(\mathcal{M}_2)= \begin{cases}
	k+3 & \text{if $k$ is odd} \\
	k+2 & \text{if $k$ is even}.
\end{cases}$$
Suppose that $\Gamma_1$, $\Gamma_2$ are the composite cycles in $D$ corresponding to $\mathcal{M}_1$, $\mathcal{M}_2$, respectively. Let $\bar{\Gamma}_1, \bar{\Gamma}_2$ be two composite cycles in $D\setminus V(\mathcal{C})$ containing only negative, positive $2$-cycles of $\Gamma$, respectively. Consider $$\Gamma_1'=\Gamma_1\bar{\Gamma}_1\quad\text{and}\quad\Gamma_2'=\Gamma_2\bar{\Gamma}_2.$$
Therefore, $\Gamma_1'$ (respectively, $\Gamma_2'$) is a composite cycle in the signed directed graph $D$ of $\mathcal{P}$ consisting only of negative (respectively, positive) $2$-cycles, and 
$$l(\Gamma_1') + l(\Gamma_2')= \begin{cases}
	m+3 & \text{if $k$ is odd} \\
	m+2 & \text{if $k$ is even},
\end{cases}$$ where $m$ is the length of maximal length cycle in $D$. It then follows from Lemma~\ref{x1} that $\mathcal{P}$ is not consistent.

\item[iii.]  Suppose that $\mathcal{C}$ has positively signed edges. Then similarly as in the proof of Theorem~\ref{x2}(iii), there exist two composite cycles
$\Gamma_1,\Gamma_2$ in $D$ and two real matrix $B_{\Gamma_1}(\epsilon) = [b_{\Gamma_1}(\epsilon)_{ij}],B_{\Gamma_2}(\epsilon) = [b_{\Gamma_2}(\epsilon)_{ij}] \in \mathcal{Q}(\mathcal{P}[\{u_1,u_2,...,u_k\}])$ with  $S_{B_{\Gamma_1}(\epsilon)}\neq S_{B_{\Gamma_2}(\epsilon)}$. Moreover, $B_{\Gamma_1}(\epsilon)$ has two algebraically simple real eigenvalues close to $\pm 1$ and the other eigenvalues lies close to zero, while the eigenvalues of $B_{\Gamma_2}(\epsilon)$ are close to the $k$-th roots of unity.

Then similarly as in (i), define two real matrices $B_{\Gamma_1'}(\epsilon), B_{\Gamma_2'}(\epsilon)\in \mathcal{Q}(\mathcal{P})$, and for $\epsilon>0$ sufficiently small, the eigenvalues of $B_{\Gamma_t'}(\epsilon)$ remain close to the eigenvalues of $B_{\Gamma_t}(\epsilon)$  for $t=1,2$, together with the second complex roots of $M^{2p}$ or $-M^{2p}$ depending on the sign of $\beta_p$, for $p=1,2,...,l$.
Therefore, $i_r(B_{\Gamma_1'}(\epsilon))\geq 2+2s+(n-m)$, and since $k$ is odd, we have $i_r(B_{\Gamma_2'}(\epsilon))= 1+2s+(n-m)$, where $s$ be the number of negative $2$-cycles in $\Gamma$.
Hence, $S_{B_{\Gamma_1'}(\epsilon)}\neq S_{B_{\Gamma_2'}(\epsilon)}$,  $\mathcal{P}$ is not consistent.
\end{itemize}

\end{proof}

The following are examples of patterns satisfying conditions (i),(ii), and (iii), respectively,  of Theorem~\ref{xx1} and so are not consistent.

\begin{eg}\label{xxeg12}\rm
Let $\mathcal{P}\in \mathcal{Q}_{10}$ be an irreducible, combinatorially symmetric sign pattern with a $0$-diagonal, whose underlying undirected graph $G$ is  given in Fig.\ref{xnfig4.1.},

\begin{figure}[H]
\centering
\tikzset{
  vertex/.style={draw, circle, inner sep=0pt, minimum size=0.15cm, fill=black},
  edge/.style={thick}
}
\begin{tikzpicture}

\node[vertex, label=above:$u_2$] (v2)  at (-2,1.25) {};
\node[vertex, label=above:$u_1$] (v1)  at (0,1.25) {};
\node[vertex, label=below:$u_3$] (v3)  at (-2,0) {};
\node[vertex, label=below:$u_4$] (v4)  at (0,0) {};

\node[vertex] (v5)  at (1.5,0) {};
\node[vertex] (v6)  at (3,0) {};
\node[vertex] (v7)  at (-3.5,0) {};
\node[vertex] (v8)  at (-5,0) {};

\node[vertex] (v14) at (-3.5,1.25) {};
\node[vertex] (v15) at (-4.75,1) {};

\draw[edge] (v1)--node[above]{$+$}(v2);
\draw[edge] (v2)--node[left]{$+$}(v3);
\draw[edge] (v3)--node[below]{$+$}(v4);
\draw[edge] (v4)--node[below]{$+$}(v5);
\draw[edge] (v1)--node[right]{$+$}(v4);

\draw[edge] (v5)--node[below]{$+$}(v6);


\draw[edge] (v3)--node[below]{$-$}(v7);
\draw[edge] (v7)--node[below]{$+$}(v8);
\draw[edge] (v7)--node[left]{$+$}(v14);
\draw[edge] (v7)--node[left]{$+$}(v15);

\end{tikzpicture}
\caption{The signed undirected graph $G$ of $\mathcal{P}$.}
\label{xnfig4.1.}
\end{figure}

Then $G$ has exactly one cycle $\mathcal{C}=u_1u_2u_3u_4$, of length $4$, which contains two nonadjacent positive edges $\{u_1,u_2\}, \{u_3,u_4\}$. The distance of the cycle $\mathcal{C}$  from any leaf is even. Thus, $\mathcal{P}$ satisfies condition (i) of Theorem~\ref{x2}. Consequently, $\mathcal{P}$ is not consistent.
\end{eg}

\begin{eg}\rm
Let $\mathcal{P}\in \mathcal{Q}_{10}$ be an irreducible, combinatorially symmetric sign pattern with a $0$-diagonal, whose underlying undirected graph $G$ is  given in Fig.\ref{xnfig4.2},

\begin{figure}[H]
\centering
\tikzset{node distance=1cm}
\begin{tikzpicture}
	\tikzstyle{vertex}=[draw, circle, inner sep=0pt, minimum size=0.15cm, fill=black]
	\tikzstyle{edge}=[thick]
	\node[vertex, label=above:$u_2$](v2) at (-2,1.25) {};
	\node[vertex, label=above:$u_1$](v1) at (0,1.25) {};
	\node[vertex, label=below:$u_3$](v3) at (-2,0) {};
	\node[vertex, label=below:$u_4$](v4) at (0,0) {};
	\node[vertex] (v5)  at (1.5,0) {};
\node[vertex] (v6)  at (3,0) {};
\node[vertex] (v7)  at (-3.5,0) {};
\node[vertex] (v8)  at (-5,0) {};

\node[vertex] (v14) at (-3.5,1.25) {};
\node[vertex] (v15) at (-4.75,1) {};

\draw[edge] (v1)--node[above]{$-$} (v2);
	\draw[edge] (v2)--node[left]{$-$} (v3);
	\draw[edge] (v3)--node[below]{$-$} (v4);
\draw[edge] (v4)--node[below]{$+$}(v5);
\draw[edge] (v1)--node[right]{$+$}(v4);

\draw[edge] (v5)--node[below]{$+$}(v6);


\draw[edge] (v3)--node[below]{$-$}(v7);
\draw[edge] (v7)--node[below]{$+$}(v8);
\draw[edge] (v7)--node[left]{$+$}(v14);
\draw[edge] (v7)--node[left]{$+$}(v15);
\end{tikzpicture}
\caption{The signed undirected graph $G$ of $\mathcal{P}$.}
\label{xnfig4.2}
\end{figure}
Then $G$ has exactly one cycle $\mathcal{C}=u_1u_2u_3u_4$, of length $4$, with more than one maximal signed path of odd length, and the distance of the cycle $\mathcal{C}$  from any leaf is even. Thus, $\mathcal{P}$ satisfies condition (ii) of Theorem~\ref{x2}. Consequently, $\mathcal{P}$ is not consistent. 
\end{eg}

\begin{eg}\rm 
Let $\mathcal{P}\in \mathcal{Q}_{9}$ be an irreducible, combinatorially symmetric sign pattern with a $0$-diagonal, whose underlying undirected graph $G$ is  given in Fig.\ref{xnfig4.1},
 

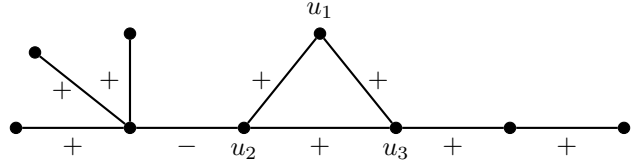
\begin{figure}[H]
\centering
\tikzset{node distance=1cm}
\begin{tikzpicture}
	\tikzstyle{vertex}=[draw, circle, inner sep=0pt, minimum size=0.15cm, fill=black]
	\tikzstyle{edge}=[thick]
	\node[vertex, label=above:$u_1$](v2) at (-1,1.25) {};
	\node[vertex, label=below:$u_2$](v3) at (-2,0) {};
	\node[vertex, label=below:$u_3$](v4) at (0,0) {};
	\node[vertex] (v5)  at (1.5,0) {};
\node[vertex] (v6)  at (3,0) {};
\node[vertex] (v7)  at (-3.5,0) {};
\node[vertex] (v8)  at (-5,0) {};

\node[vertex] (v14) at (-3.5,1.25) {};
\node[vertex] (v15) at (-4.75,1) {};

	\draw[edge] (v2)--node[left]{$+$} (v3);
	\draw[edge] (v2)--node[right]{$+$} (v4);
	\draw[edge] (v3)--node[below]{$+$} (v4);
    
	\draw[edge] (v4)--node[below]{$+$}(v5);

\draw[edge] (v5)--node[below]{$+$}(v6);


\draw[edge] (v3)--node[below]{$-$}(v7);
\draw[edge] (v7)--node[below]{$+$}(v8);
\draw[edge] (v7)--node[left]{$+$}(v14);
\draw[edge] (v7)--node[left]{$+$}(v15);
\end{tikzpicture}
\caption{The signed undirected graph $G$ of $\mathcal{P}$.}
\label{xnfig4.1}
\end{figure}
Then $G$ contains exactly one cycle, namely $\mathcal{C}=u_1u_2u_3$, of length $3$ and whose edges are all positive. The distance of the cycle $\mathcal{C}$  from any leaf is even. Thus, $\mathcal{P}$ satisfies condition (iii) of Theorem~\ref{x2}. Consequently, $\mathcal{P}$ is not consistent.
\end{eg}

Next, we study irreducible, combinatorially symmetric sign patterns with more than one cycle in their underlying signed undirected graph. The following definition is from \cite{2025}.

\begin{defn}[Definition~3.1, \cite{2025}]
Suppose that $\mathcal{P}$ is an irreducible, combinatorially symmetric sign pattern whose underlying signed undirected graph is $G$. Then any two cycles $\mathcal{C}_1$, $\mathcal{C}_2$ in $G$ are called \textit{path-adjacent} if there exists a path $uv_1v_2\cdots v_{r}w$ in $G$ where $u\in V(\mathcal{C}_1)$, $w\in V(\mathcal{C}_2)$ and $v_1,v_2,...,v_{r}$ are not vertices of any cycle in $G$.
\end{defn}

\begin{lemma}[Lemma~3.3, \cite{2025}]\label{nl111}
Let $\mathcal{P}\in \mathcal{Q}_n$ be an irreducible, combinatorially symmetric sign pattern with a $0$-diagonal, whose underlying signed undirected graph is $G$ and the signed directed graph is $D$. Suppose $G$ contains at least one cycle, has no leaf, and let the distance between any two path-adjacent cycles in $G$ be odd. Let $\mathcal{C}=u_1u_2\cdots u_k$ be any cycle in $G$ where $k\geq 3$, then $D$ has a composite cycle of length $n$ containing the directed cycle $(u_1,u_2,...,u_k)$.
\end{lemma}

\begin{theorem}\label{xx2}
Let $\mathcal{P}\in \mathcal{Q}_n$ be an irreducible, combinatorially symmetric sign pattern with a $0$-diagonal, whose underlying signed $G$ has at least one cycle, has no leaf, and let the distance between any two path-adjacent cycles in $G$ be odd. Then $\mathcal{P}$ is not consistent if any one of the following holds.
\begin{itemize}
\item[i.] $G$ has a cycle  $\mathcal{C}$ with at least two nonadjacent positively signed edges.
\item[ii.] $G$ has a cycle $\mathcal{C}$ with more than one maximal signed path of odd length.
\item[iii.] $G$ contains a cycle $\mathcal{C}$ of odd length, which has at least one positive edge.
\end{itemize}
\end{theorem} 
\begin{proof} 
Assume that $G$ contains exactly one cycle, then according to Theorem~\ref{x2}, we can conclude that $\mathcal{P}$ is not consistent. Suppose that $G$ has more than one cycle, and let $\mathcal{C} = u_1u_2\cdots u_k$ be a cycle in $G$. Let $D$ denote the signed directed graph associated with $G$. 
Consider a composite cycle $\Gamma = \beta_1 \beta_2 \cdots \beta_t$ that has the maximum length in $D \setminus V(\mathcal{C})$, where $l(\Gamma) = l$ and each $\beta_r$ is a simple cycle in $D$ for integers $1 \leq r \leq t$. According to Lemma~\ref{nl111}, then have $n = k + l$. 
\begin{itemize}
\item[i.] Similarly as in the proof of Theorem~\ref{x2}(i), there exist two composite cycles
$\Gamma_1,\Gamma_2$ in $D$ and two real matrix $B_{\Gamma_1}(\epsilon)=[b_{\Gamma_1}(\epsilon)_{ij}],B_{\Gamma_2}(\epsilon)=[b_{\Gamma_2}(\epsilon)_{ij}] \in \mathcal{Q}(\mathcal{P}[\{u_1,u_2,...,u_k\}])$ with $S_{B_{\Gamma_1}(\epsilon)}\neq S_{B_{\Gamma_2}(\epsilon)}$. Moreover, $B_{\Gamma_1}(\epsilon)$ has four algebraically simple real eigenvalues close to $\pm 10$ and $\pm 10^2$ and the other eigenvalues lies close to zero, while the eigenvalues of $B_{\Gamma_2}(\epsilon)$ are close to the $k$-th roots of unity. Let,
$$M>\max\{1,|\lambda|: \lambda~ \text{is an eigenvalue of either} ~B_{\Gamma_1}(\epsilon)~\text{or}~B_{\Gamma_2}(\epsilon)\}.$$ 
Consider $$\Gamma_1'=\Gamma_1\Gamma\quad\text{and}\quad\Gamma_2'=\Gamma_2\Gamma$$
and define the real matrices $B_{\Gamma_r'}(\epsilon)=[b_{\Gamma_r'}(\epsilon)_{ij}]\in \mathcal{Q}(\mathcal{P})$ for $r=1,2$, as
$$|b_{\Gamma_r'}(\epsilon)_{ij}|= \begin{cases}
	|b_{\Gamma_r}(\epsilon)_{ij}| & \text{if}\ p_{ij}~ \text{is in } \Gamma_r \\
	M^p & \text{if}\ p_{ij}~ \text{is in}~ \beta_p,~\text{for}~p=1,2,...,t \\
	\epsilon & \text{if}~ p_{ij}\neq 0~\text{and not in}~\Gamma_r, \Gamma\\
	0 & \text{elsewhere.}
\end{cases}$$
For $\epsilon > 0$ sufficiently small , the eigenvalues of $B_{\Gamma_r'}(\epsilon)$ remain close to the eigenvalues of $B_{\Gamma_r}(\epsilon)$  for $r=1,2$, together with the $l(\beta_p)$-th complex roots of $M^{pl(\beta_p)}$, $-M^{pl(\beta_p)}$, depending on the sign of $\beta_p$, for $p=1,2,...,t$. 

Suppose that $B_{\Gamma}(\epsilon)\in \mathcal{Q}(\mathcal{P}[\{u_1,u_2,...,u_k\}^c])$ is the submatrix of both $B_{\Gamma_1'}(\epsilon), B_{\Gamma_2'}(\epsilon)$. Therefore, $i_r(B_{\Gamma_1'}(\epsilon))\geq 4+i_r(B_{\Gamma}(\epsilon))$ and $i_r(B_{\Gamma_2'}(\epsilon))\leq 2+i_r(B_{\Gamma}(\epsilon))$. 
Thus $S_{B_{\Gamma_1'}(\epsilon)} \neq S_{B_{\Gamma_2'}(\epsilon)}$, and hence $\mathcal{P}$ is not consistent.

\item[ii.] Similarly as in the proof of Theorem~\ref{x2}(ii) that there exist two matchings $\mathcal{M}_1=\{\alpha_{i_1},\alpha_{i_2},..., \alpha_{i_{k_1}}\}$ and $\mathcal{M}_2=\{\alpha_{j_1},\alpha_{j_2},..., \alpha_{j_{k_2}}\}$ in $G$ such that $\mathcal{M}_1$ contains only positive edges and $\mathcal{M}_2$ contains only negative edges of $\mathcal{C}$, satisfying $2k_1 + 2k_2 \geq k + 2$, where $\alpha_t$ for $t=1,2,...,k$ are the edges of $\mathcal{C}$. Let $\Gamma_1$, $\Gamma_2$ denote the composite cycles in $D$ corresponding to $\mathcal{M}_1$, $\mathcal{M}_2$, respectively. Consider
 $$\Gamma_1'=\Gamma_1\Gamma\quad\text{and}\quad\Gamma_2'=\Gamma_2\Gamma.$$ Define real matrices $B_{\Gamma_1'}(\epsilon)=[b_{\Gamma_1'}(\epsilon)_{ij}]\in \mathcal{Q}(\mathcal{P})$ as
$$|b_{\Gamma_1'}(\epsilon)_{ij}|= \begin{cases}
	10^p & \text{if}\ p_{ij}~ \text{is}~ \text{in}~ \alpha_p,~\text{for}~p=i_1,i_2,...,i_{k_1} \\
	11^p & \text{if}\ p_{ij}~ \text{is}~ \text{in}~ \beta_p,~\text{for}~p=1,2,...,t \\
	\epsilon & \text{if}~ p_{ij}\neq 0~\text{and not in both}~ \Gamma_1,~ \Gamma\\
	0 & \text{elsewhere.}
\end{cases}$$ Similarly, define the real matrix $B_{\Gamma_2'} (\epsilon)\in \mathcal{Q}(\mathcal{P})$.
Suppose that $B_{\Gamma}(\epsilon)\in \mathcal{Q}(\mathcal{P}[\{u_1,u_2,...,u_k\}^c])$ is a submatrix of both $B_{\Gamma_1'}(\epsilon)$, $B_{\Gamma_2'}(\epsilon)$.
Since $\Gamma_1$ contains only negative $2$-cycles. Hence, for $\epsilon > 0$ sufficiently small, $$i_r(B_{\Gamma_1'}(\epsilon))\geq 2k_1+i_r(B_{\Gamma}(\epsilon))
.$$

 Also, $\Gamma_2$ contains only positive $2$-cycles, for $\epsilon > 0$ sufficiently small $$i_c(B_{\Gamma_1'}(\epsilon))\geq 2k_2+i_c(B_{\Gamma}(\epsilon)).$$
Therefore,
$$i_r(B_{\Gamma_1'}(\epsilon))+i_c(B_{\Gamma_2'}(\epsilon))\geq (2k_1+i_r(B_{\Gamma}(\epsilon)))+(2k_2+i_c(B_{\Gamma}(\epsilon)))\geq n+2,$$
since $i_l(B_{\Gamma}(\epsilon)))+i_c(B_{\Gamma}(\epsilon))=l$ and $n=k+l$.
Hence, by Remark \ref{xx5}, $\mathcal{P}$ is not consistent.

\item[iii.] Suppose that $\mathcal{C}$ has positively signed edges. Then similarly as in the proof of Theorem~\ref{x2}(iii), there exist two composite cycles
$\Gamma_1,\Gamma_2$ in $D$ and two real matrix $B_{\Gamma_1}(\epsilon) = [b_{\Gamma_1}(\epsilon)_{ij}],B_{\Gamma_2}(\epsilon) = [b_{\Gamma_2}(\epsilon)_{ij}] \in \mathcal{Q}(\mathcal{P}[\{u_1,u_2,...,u_k\}])$ with  $S_{B_{\Gamma_1}(\epsilon)}\neq S_{B_{\Gamma_2}(\epsilon)}$. Moreover, $B_{\Gamma_1}(\epsilon)$ has two algebraically simple real eigenvalues close to $\pm 1$ and the other eigenvalues lies close to zero, while the eigenvalues of $B_{\Gamma_2}(\epsilon)$ are close to the $k$-th roots of unity.

Then similarly as in part~(i), define two real matrices $B_{\Gamma_1'}(\epsilon), B_{\Gamma_2'}(\epsilon)\in \mathcal{Q}(\mathcal{P})$, and for $\epsilon>0$ sufficiently small, the eigenvalues of $B_{\Gamma_r'}(\epsilon)$ remain close to the eigenvalues of $B_{\Gamma_r}(\epsilon)$  for $r=1,2$, together with the $l(\beta_p)$-th complex roots of $M^{pl(\beta_p)}$, $-M^{pl(\beta_p)}$, depending on the sign of $\beta_p$, for $p=1,2,...,t$.
Hence, $S_{B_{\Gamma_1'}(\epsilon)}\neq S_{B_{\Gamma_2'}(\epsilon)}$,  $\mathcal{P}$ is not consistent.
\end{itemize}

\end{proof}

The following are examples of patterns satisfying conditions (i), (ii), and (iii), respectively, of Theorem~\ref{xx2}. Hence, they are not consistent.

\begin{eg}\label{egx1.4}\rm
Let $\mathcal{P}\in \mathcal{Q}_7$ be an irreducible, combinatorially symmetric sign pattern with a $0$-diagonal, whose underlying undirected graph $G$ is  given in Fig.\ref{xnfig5.1},

\begin{figure}[H]
\centering
\tikzset{node distance=1cm}
\begin{tikzpicture}
	\tikzstyle{vertex}=[draw, circle, inner sep=0pt, minimum size=0.15cm, fill=black]
	\tikzstyle{edge}=[thick]
	\node[vertex, label=above:$u_1$](v1) at (0,1.25) {};
	\node[vertex, label=above:$u_2$](v2) at (-2,1.25) {};
	\node[vertex, label=below:$u_3$](v3) at (-2,0) {};
	\node[vertex, label=below:$u_4$](v4) at (0,0) {};
	\node[vertex, label=below:$u_5$](v5) at (1.5,0) {};
	\node[vertex, label=below:$u_6$](v6) at (3.5,0) {};
	\node[vertex, label=above:$u_7$](v7) at (2.5,1.25) {};
	
	\draw[edge] (v1)--node[above]{$+$} (v2);
	\draw[edge] (v2)--node[left]{$+$} (v3);
	\draw[edge] (v3)--node[below]{$+$} (v4);
	\draw[edge] (v4)--node[right]{$+$} (v1);
	\draw[edge] (v4)--node[below]{$+$} (v5);
	\draw[edge] (v5)--node[below]{$+$} (v6);
	\draw[edge] (v7)--node[right]{$+$} (v6);
	\draw[edge] (v7)--node[left]{$+$} (v5);

\end{tikzpicture}
\caption{The signed undirected graph $G$ of $\mathcal{P}$.}
\label{xnfig5.1}
\end{figure}
Then $G$ has a cycle $\mathcal{C}=u_1u_2u_3u_4$ with two nonadjacent positively signed edges $\{u_1,u_2\},\{u_3,u_4\}$, and the distance from the cycle $\mathcal{C}$ to $\mathcal{C}_1=u_5u_6u_7$ is odd. Thus, $\mathcal{P}$ satisfies condition (i) of Theorem~\ref{xx2}, which implies that $\mathcal{P}$ is not consistent.
\end{eg}

\begin{eg}\rm
Let $\mathcal{P}\in \mathcal{Q}_8$ be an irreducible, combinatorially symmetric sign pattern with a $0$-diagonal, whose underlying undirected graph $G$ is  given in Fig.\ref{xnfig6.1},

\begin{figure}[H]
\centering
\tikzset{node distance=1cm}
\begin{tikzpicture}
	\tikzstyle{vertex}=[draw, circle, inner sep=0pt, minimum size=0.15cm, fill=black]
	\tikzstyle{edge}=[thick]
	\node[vertex, label=above:$u_1$](v1) at (0,1.25) {};
	\node[vertex, label=above:$u_2$](v2) at (-2,1.25) {};
	\node[vertex, label=below:$u_3$](v3) at (-2,0) {};
	\node[vertex, label=below:$u_4$](v4) at (0,0) {};
	\node[vertex, label=below:$u_5$](v5) at (1.5,0) {};
	\node[vertex, label=below:$u_6$](v6) at (3.5,0) {};
	\node[vertex, label=above:$u_7$](v7) at (3.5,1.25) {};
	\node[vertex, label=above:$u_8$](v8) at (1.5,1.25) {};
	
	\draw[edge] (v1)--node[above]{$-$} (v2);
	\draw[edge] (v2)--node[left]{$-$} (v3);
	\draw[edge] (v3)--node[below]{$-$} (v4);
	\draw[edge] (v4)--node[right]{$+$} (v1);
	\draw[edge] (v4)--node[below]{$+$} (v5);
	\draw[edge] (v5)--node[below]{$+$} (v6);
	\draw[edge] (v7)--node[right]{$+$} (v6);
	\draw[edge] (v7)--node[above]{$+$} (v8);
	\draw[edge] (v8)--node[left]{$+$} (v5);
\end{tikzpicture}
\caption{The signed undirected graph $G$ of $\mathcal{P}$.}
\label{xnfig6.1}
\end{figure}
Then $G$ has a cycle $\mathcal{C}=u_1u_2u_3u_4$ with more than one maximal signed path of odd length, and the distance from cycle $\mathcal{C}$ to $\mathcal{C}_1=u_5u_6u_7u_8$ is odd. Thus, $\mathcal{P}$ satisfies condition (ii) of Theorem~\ref{xx2}, which implies that $\mathcal{P}$ is not consistent.
\end{eg}

\begin{eg}\rm
Let $\mathcal{P}\in \mathcal{Q}_7$ be an irreducible, combinatorially symmetric sign pattern with a $0$-diagonal, whose underlying undirected graph is $G$ as in Example~\ref{egx1.4}. Then $G$ has a cycle $\mathcal{C}_1=u_5u_6u_7$ of odd length with positively signed edges, and the distance from the cycle $\mathcal{C}_1$ to $\mathcal{C}=u_1u_2u_3u_4$ is odd. Thus, $\mathcal{P}$ satisfies condition (iii) of Theorem~\ref{xx2}, which implies that $\mathcal{P}$ is not consistent.
\end{eg}

\section{
$2$-consistent sign patterns} \label{s5}
In \cite{1993a}, the author characterized the class $\pi$ of all sign patterns that require exactly one real eigenvalue. In this section, we define the class $\Delta$ of all $2$-consistent sign pattern matrices, and we derive certain necessary conditions for sign patterns to be in $\Delta$. As explained earlier, we consider only irreducible sign patterns in $\Delta$.
Before we proceed, we clarify our convention on the sign of a cycle.

\begin{remark}\rm
    Our definition of the sign of a cycle differs from that in \cite{1993a}.  
In \cite{1993a}, the sign of a simple cycle of length $k$ is defined as the product of the sign of the edge, whereas in this paper, we define it as $(-1)^{k-1}$ times this product. Throughout this work, we adopt our definition. Whenever we cite results from~\cite{1993a}, we use the definition of the sign of a cycle as given in \cite{1993a}. 
\end{remark}

The following result is by Eschenbach \cite{1993a},

\begin{lemma}[Lemma 1.3, \cite{1993a}]
	If $\mathcal{A}$ is an $n\times n$ sign pattern matrix in $\pi$, and if $\gamma=a_{i_1i_2}a_{i_2i_3}\cdots a_{i_ki_1}$ is a nonzero odd cycle in $\mathcal{A}$, then $\mathcal{A}[\{i_1i_2\cdots i_k\}^c]$ contains no nonzero odd cycles.
\end{lemma}

We have a similar necessary condition for sign patterns in $\Delta$.
\begin{lemma} \label{l2}
	Suppose that $\mathcal{P}\in \mathcal{Q}_n$ is a sign pattern matrix in $\Delta$, 
and let $\gamma=(i_1,i_2,\dots,i_k)$ be a negative even cycle in the signed directed graph $D$ of $\mathcal{P}$. Then the signed directed graph $D'$ corresponding to $\mathcal{P}[\{i_1 i_2 \cdots i_k\}^c]$ contains neither a negative even cycle nor a nonzero odd cycle.
\end{lemma}
\begin{proof}
	Assume that
$
\alpha=(i_{k+1},i_{k+2},\dots,i_{k+j})
$
is either a negative even cycle or a nonzero odd cycle in $D'$.
	Consider the composite cycle $\Gamma=\gamma\alpha$, and define $B_{\Gamma}(0)=(b_{\Gamma}(0)_{ij})$ as in \eqref{xe3}. 
	Then $B_\Gamma(0)$ has $k$ nonzero algebraic simple eigenvalues equal to the $k$-th complex root of $10^{k}$, and $j$ nonzero algebraic simple eigenvalues equal to the $j$-th complex root of $10^{2j}$ (respectively, $10^{2j}$ or $-10^{2j}$) if $j$ is even (respectively, odd). Take $B_{\Gamma}(\epsilon)=(b_{\Gamma}(\epsilon)_{ij})$ as in \eqref{e2}. Then $B_{\Gamma}(\epsilon)\in \mathcal{Q}(\mathcal{P})$ and for $\epsilon > 0$ sufficiently small, it has at least three real eigenvalues close to $10$, $-10$, $10^2$ and $-10^2$ depending on $j$ and sign of $\alpha$, which contradicts the assumption that $\mathcal{P}\in \Delta$. 
\end{proof}

	

\begin{lemma} \label{l3}
	Suppose that $\mathcal{P}\in  \mathcal{Q}_n$ is a sign pattern matrix in $\Delta$, then the signed directed graph $D$ of $\mathcal{P}$ contains no nonzero composite cycle of length $n$ consisting only of positive even cycles.
\end{lemma}
\begin{proof}
	Suppose that $D$ containing the nonzero composite cycle $\Gamma=\gamma_1 \gamma_2 \cdots \gamma_k$ of length $n$, where $\gamma_p$ for $p=1,2,...,k$, is a positive cycles of even length. Define the real matrix $B_{\Gamma}(0)=(b_{\Gamma}(0)_{ij})$ as in \eqref{xe3}.
	Also, for $\epsilon>0$, consider $B_\Gamma(\epsilon)=(b_\Gamma(\epsilon)_{ij})$ as in \eqref{e2} by replacing $\gamma$ with $\Gamma$, then $B_\Gamma(\epsilon)\in \mathcal{Q}(\mathcal{P})$. Since $\Gamma$ contains only positive even cycles, $B_\Gamma(0)$ has $n$ algebraically simple nonreal eigenvalues. Thus, the perturbed matrix $B_\Gamma(\epsilon)$ has $n$ nonreal eigenvalues close to the $n$ distinct eigenvalues of $B_\Gamma(0)$, which contradicts that $\mathcal{P}\in \Delta$. \end{proof}

    In addition to the above conditions, which are useful for identifying sign patterns which are not in $\Delta$, we further give useful necessary conditions of $2$-consistent sign patterns in $\Delta$, specific to sign singular, sign nonsingular, and sign patterns allowing singularity, respectively.

The following necessary and sufficient conditions for a sign singular sign pattern to be in $\pi$ were established by Eschenbach~\cite{1993a}.

\begin{theorem}[Theorem 1.4, \cite{1993a}]
	If $\mathcal{A}$ is an $n\times n$ sign singular matrix, then $\mathcal{A}\in \pi$ if and only if \begin{itemize}
		\item[1.] there are no nonzero odd cycles in $\mathcal{A}$;
		\item[2.] all even cycles in $\mathcal{A}$ are nonpositive; and \item[3.] among $p\times q$ $0$-submatrices of $\mathcal{A}$, $\max(p+q)\leq n+1$.
	\end{itemize}
\end{theorem}
The following are similar necessary conditions for a sign singular sign pattern to be in $\Delta$.
\begin{theorem} \label{4.7l}
	Let $\mathcal{P}\in \mathcal{Q}_n$ be a sign singular matrix with underlying signed directed graph $D$. If $\mathcal{P}\in \Delta$, then the following conditions hold.
   \begin{itemize}
		\item[1.] All even cycles in $D$ are nonnegative.
		\item[2.] Among $p\times q$ $0$-submatrices of $\mathcal{P}$, $\max(p+q)\leq n+2$.
		\item[3.] If $\gamma=(i_1,i_2,...,i_k)$ is a nonzero odd cycle in $D$, then the signed directed graph corresponding to $\mathcal{P}[\{i_1i_2\cdots i_k\}^c]$ does not have a nonzero odd cycle.
		\item[4.] If there is no composite cycle of length $n-1$ in $D$, then all cycles in $D$ are even and nonnegative. Furthermore, there is a composite cycle in $D$ of length $n-2$.
		\item[5.] If $D$ contains a composite cycle of length $n-1$, and all such composite cycles of length $n-1$ in $D$ have the same sign, then every odd cycle in $D$ must have the same sign.
\item[6.] If there are two composite cycles of length $n-1$ in $D$ having opposite sign, then there is a composite cycle of length $n-2$ in $D$. 
	\end{itemize}
\end{theorem}

\begin{proof}
	Since $\mathcal{P}=[p_{ij}]$ is sign singular, every $B\in \mathcal{Q}(\mathcal{P})$ has an eigenvalue equal to $0$. Consider $$\Ch_B(x)=x(x^{n-1}+E_1(B)x^{n-2}+\cdots+E_{n-1}(B)).$$
	
	\begin{itemize}
		\item[1.]  Assume that $\mathcal{P}$ has a negative even cycle $\gamma$. Define the real matrix $B_{\gamma}(0)=(b_{\gamma}(0)_{ij})$ as in \eqref{e1} and the perturbed matrix $B_{\gamma}(\epsilon)=(b_{\gamma}(\epsilon)_{ij})$ as in \eqref{e2} for some $\epsilon>0$. Then $B_{\gamma}(\epsilon)\in \mathcal{Q}(\mathcal{P})$  and it has real eigenvalues close to $1$ and $-1$. Also, $B_{\gamma}(\epsilon)$ is a singular matrix, so $0$ is an eigenvalue of $B_{\gamma}(\epsilon)$ which contradicts that $\mathcal{P}\in \Delta$.
		\item[2.] Suppose $\mathcal{P}\in \Delta$ contains a $p\times q$ zero submatrix with $p+q\geq n+3$. Then, by Theorem~3 in \cite{1993}, $\mathcal{P}$ requires the eigenvalue $0$ with algebraic multiplicity $k\geq 3$, which contradicts the assumption that $\mathcal{P}\in \Delta$.
		\item[3.]  Assume that $\alpha=(i_{k+1}, i_{k+2},..., i_{k+j})$ is an odd cycle in the underlying directed graph corresponding to $\mathcal{P}[\{i_1i_2\cdots i_k\}^c]$. Consider  $\Gamma=
        \gamma\alpha$ and define $B_{\Gamma}(0)=(b_{\Gamma}(0)_{ij})$ as in \eqref{xe3}.
		Then $B_\Gamma(0)$ has a real eigenvalue equal to $10$ or $-10$ depending on the sign of $\gamma$ and has a real eigenvalue equal to $10^2$ or $-10^2$ depending on the sign of $\alpha$. For some $\epsilon>0$, define $B_{\Gamma}(\epsilon)=(b_{\Gamma}(\epsilon)_{ij})$ as in \eqref{e2} with $\gamma$ replaced by $\Gamma$. Then $B_{\Gamma}(\epsilon)\in \mathcal{Q}(\mathcal{P})$ and for $\epsilon > 0$ sufficiently small, it has real eigenvalues close to $10$ and $10^2$ in absolute value. Also, $B_{\Gamma}(\epsilon)$ is a singular matrix, so $0$ is an eigenvalue of $B_{\Gamma}(\epsilon)$, which contradicts that $\mathcal{P}\in \Delta$. 
		
		\item[4.] If $D$ has no nonzero composite cycle of length $n-1$, then the characteristic polynomial of any $B\in \mathcal{Q}(\mathcal{P})$ is $$\Ch_B(x)=x^2(x^{n-2}+E_1(B)x^{n-3}+\cdots+E_{n-2}(B)).$$ Therefore, $0$ is an eigenvalue of $B$ with algebraic multiplicity $2$, for all $B\in \mathcal{Q}(\mathcal{P})$. Suppose $D$ contains an odd cycle or a negative even cycle $\gamma$. Then define the real matrix $B_\gamma(\epsilon)$ as in \eqref{e2} which lies in $\mathcal{Q}(\mathcal{P})$ and for $\epsilon > 0$ sufficiently small, it has at least one nonzero real eigenvalue, which contradicts that $\mathcal{P}\in \Delta$. 
        
		Furthermore, if $D$ has no composite cycle of length $n-1$ and $n-2$, then $0$ is an eigenvalue of $B$ with algebraic multiplicity greater than equal to $3$ for all $B\in \mathcal{Q}(\mathcal{P})$, which contradicts that $\mathcal{P}\in \Delta$.
		
		\item[5.] 
        Suppose that $D$ contains two oppositely signed odd cycles $\gamma_1$ and $\gamma_2$ of length $k_1$ and $k_2$, respectively. Without loss of generality, let $\gamma_1$ be positive and $\gamma_2$ be negative. Define $B_{\gamma_1}(0)=[b_{\gamma_1}(0)_{ij}]$ as in \eqref{e1}, with $\gamma_1$ replacing $\gamma$. Then for $\epsilon > 0$ sufficiently small, the perturbed matrix $B_{\gamma_1}(\epsilon)=[b_{\gamma_1}(\epsilon)_{ij}]$, as defined in \eqref{e2} is in $\mathcal{Q}(\mathcal{P})$ and has real eigenvalues $0$ and $\lambda$, where $\lambda$ is a real number close to $1$. Since $\mathcal{P}\in \Delta$, all other eigenvalues of $B_{\gamma_1}(\epsilon)$ are nonreal and occur in conjugate pairs $\alpha_i$, $\bar{\alpha}_i$ for $i=1,2,...,\frac{n-2}{2}$. Hence, $$E_{n-1}(B_{\gamma_1}(\epsilon))=\lambda   \alpha_1 \bar{\alpha}_1\cdots \alpha_{\frac{n-2}{2}}  \bar{\alpha}_{\frac{n-2}{2}}>0.$$ 
		Similarly, let $B_{\gamma_2}(0)$ be a real matrix as defined in \eqref{e1} with $\gamma_2$ replacing $\gamma$. Then for $\epsilon > 0$ sufficiently small,  the perturbed matrix $B_{\gamma_2}(\epsilon)$, as defined in \eqref{e2} is in $\mathcal{Q}(\mathcal{P})$ and has real eigenvalues $0$ and $\lambda$, where $\lambda$ is a real number close to $-1$. Hence, $$E_{n-1}(B_{\gamma_2}(\epsilon))=\lambda  \alpha_1 \bar{\alpha}_1\cdots \alpha_{\frac{n-2}{2}}  \bar{\alpha}_{\frac{n-2}{2}}<0,$$ 
		which contradicts the assumption.

        \item[6.] Let $\Gamma_1$ and $\Gamma_2$ be two oppositely signed composite cycles in $D$ of length $n-1$. 
For $\epsilon, t>0$, define a real matrix $B(t)=[b_{ij}(t)]\in \mathcal{Q}(\mathcal{P})$ by  
\[
|b_{ij}(t)| =
\begin{cases}
	t & \text{if } p_{ij}\ \text{lies in } \Gamma_1 \setminus \Gamma_2 \\
	\frac{1}{t} & \text{if } p_{ij}\ \text{lies in } \Gamma_2 \setminus \Gamma_1\\
    1 & \text{if } p_{ij}\ \text{lies in both } \Gamma_1 \text{ and } \Gamma_2 \\
	\epsilon & \text{if } p_{ij}\neq 0 \text{ and lies in neither } \Gamma_1 \text{ nor } \Gamma_2 \\
	0 & \text{otherwise.}
\end{cases}
\]

For $\epsilon>0$ sufficiently small, there exist $t_1,t_2>0$ such that $
E_{n-1}(B(t_1)), E_{n-1}(B(t_2))$ have opposite sign. 
Since $E_{n-1}(B(t))$ is a continuous function of $t$, there exists $t_3>0$ such that $E_{n-1}(B(t_3))=0$.

If $D$ has no composite cycle of length $n-2$, then $0$ is an eigenvalue of $B(t_3)$ with algebraic multiplicity at least $3$, which contradicts the assumption that $\mathcal{P}\in \Delta$.      
	\end{itemize}
\end{proof}

However, the converse of Theorem \ref{4.7l} is not true, that is, the stated conditions are not sufficient for a sign pattern to be in $\Delta$.
\begin{eg} \rm
 Let $\mathcal{P}\in \mathcal{Q}_n$ be a sign pattern, whose underlying signed directed graph is $D$  given in Fig.\ref{figx5.6}.
		\begin{figure}[H]
\centering
\begin{minipage}{.45\textwidth}
\vspace{-1cm}
$$
\mathcal{P}=\begin{bmatrix}
		0&0&0&-\\+&-&+&0\\0&-&0&-\\0&+&0&0
	\end{bmatrix}
$$
\end{minipage}
\hspace{0.05\textwidth}
\begin{minipage}{.45\textwidth}

\centering
\usetikzlibrary{decorations.markings}

\tikzset{
  vertex/.style={
    draw, circle, fill=black,
    inner sep=0pt, minimum size=0.15cm
  },
  edge/.style={
    thick,
    postaction={
      decorate,
      decoration={
        markings,
        mark=at position 0.75 with {\arrow{stealth}}
      }
    }
  },
  loopedge/.style={thick, ->}
}

\begin{tikzpicture}[scale=1, every node/.style={font=\small}]

\node[vertex,label=above:$1$] (v1) at (0,1.5) {};
\node[vertex,label=above:$2$] (v2) at (1.5,0) {};
\node[vertex,label=below:$3$] (v3) at (0,-1.5) {};
\node[vertex,label=left:$4$]  (v4) at (-1.5,0) {};

\draw[edge] (v2) -- node[right] {$+$} (v1);
\draw[edge] (v2) -- node[above] {$+$} (v3);
\draw[edge] (v3) -- node[below] {$-$} (v4);
\draw[edge] (v1) -- node[above] {$-$} (v4);
\draw[edge] (v4) -- node[below] {$+$} (v2);

\draw[edge, bend right=25] (v3) to node[right] {$-$} (v2);

\draw[loopedge] (v2) edge[loop right] node {$-$} (v2);

\end{tikzpicture}
\caption{The signed directed graph $D$.}
\label{figx5.6}

\end{minipage}
\end{figure}
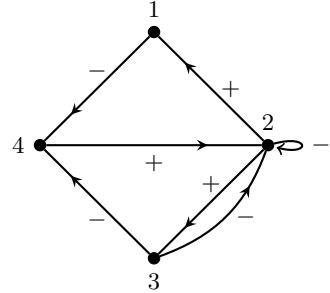 Clearly $\mathcal{P}$ is sign singular.
 Also, $D$ does not contain any negative even cycle or vertex disjoint odd cycles. Also, $D$ contains at least one cycle of length $3$, and all the odd cycles, namely $(1,4,2),(2,3,4),(2,2)$ have the same sign. Moreover, among all zero submatrices of $\mathcal{P}$ of order $p \times q$, we have $\max(p+q) = 6$.
 Therefore, $\mathcal{P}$ satisfies all the conditions stated in Theorem \ref{4.7l}. However
	$$B=\begin{bmatrix}
		0&0&0&-1\\3&-6&1&0\\0&-11&0&-3\\0&1&0&0
	\end{bmatrix}\in \mathcal{Q}(\mathcal{P})$$ has eigenvalues $0,-1,-2,-3$, hence $\mathcal{P}\notin \Delta$. 
\end{eg}

Eschenbach~\cite{1993a} established the following result for sign patterns that allow singularity.
\begin{theorem}[Theorem 1.5, \cite{1993a}] \label{l4.2}
	If $\mathcal{A}$ is an $n\times n$ sign pattern matrix that allows singularity, then $\mathcal{A}\in \pi$ if and only if \begin{itemize}
		\item[1.] all odd cycles of length $l<n$ in $\mathcal{A}$ are zero;
		\item[2.] $\mathcal{A}$ contains oppositely signed $n$-cycles;
		\item[3.] $\mathcal{A}$ contains a principal matching of size $n-1$; and 
		\item[4.] all even cycles in $\mathcal{A}$ are nonpositive.
	\end{itemize}
\end{theorem}

The following theorem gives some necessary conditions for sign patterns that allow singularity to be in $\Delta$.

\begin{theorem} \label{4.8l}
	If $\mathcal{P}\in \mathcal{Q}_n$ is a sign pattern matrix that allows singularity (but is not sign singular), whose underlying signed directed graph is $D$. If $\mathcal{P}\in \Delta$, then the following conditions hold.
    \begin{itemize}
		\item[1.] 
        $D$ contains oppositely signed composite cycles of length $n$.
		\item[2.] If there is no composite cycle in $D$ of length $n-1$, then there is a composite cycle in $D$ of length $n-2$.
        \item[3.] $D$ contains either a negative even cycle or two vertex disjoint odd cycles of opposite sign.
       
		\item[4.] $D$ contains two nonzero odd cycles of the same sign, which are vertex disjoint.
		 
	\end{itemize} 
\end{theorem}
\begin{proof}
	\begin{itemize}
		\item[1.] Since $\mathcal{P}$ allows singularity (but is not sign singular), there exist two oppositely signed composite cycles of length $n$.
		
		\item[2.] Suppose that $D$ contains no composite cycle of length $n-1$. Since $\mathcal{P}$ allows singularity, there exists a singular matrix $B\in \mathcal{Q}(\mathcal{P})$. Therefore, the characteristic polynomial of $B$ is  
$$
\Ch_B(x) = x^2\big(x^{\,n-2} + E_1(B)x^{\,n-3} + \cdots + E_{n-2}(B)\big).
$$
If $D$ contains no composite cycle of length $n-2$, then $E_{n-2}(B)=0$, which implies that $0$ is an eigenvalue of $B$ with algebraic multiplicity $3$. This contradicts the assumption that $\mathcal{P}\in \Delta$.

		\item[3.] Since $\mathcal{P}$ allows singularity, it follows from part~(1) that $D$ contains a negatively signed composite cycle $\Gamma$ of length $n$.
By Lemmas~\ref{l2} and~\ref{l3}, $\Gamma$ contains either exactly one negative even cycle or exactly two nonzero odd cycles of opposite sign.

\item[4.] Since $\mathcal{P}$ allows singularity, it follows from part~(1) that $D$ contains a positively signed composite cycle $\Gamma$ of length $n$.
By Lemma~\ref{l2} and~\ref{l3}, $\Gamma$ must contain a nonzero odd cycle. Moreover, since $n$ is even, $\Gamma$ necessarily contains two vertex-disjoint odd cycles of the same sign. 

	\end{itemize}
\end{proof}
However, the following example shows that the converse of Theorem \ref{4.8l} is not true.
\begin{eg}\rm
 Let $\mathcal{P}\in \mathcal{Q}_n$ be a sign pattern, whose underlying signed directed graph is $D$  given in Fig.\ref{figx5.10}.
		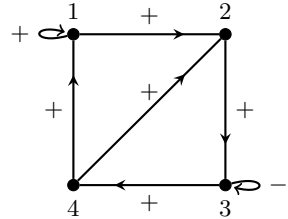
\begin{figure}[H]
\centering
\begin{minipage}{.45\textwidth}
\vspace{-1cm}
$$
\mathcal{P}=\begin{bmatrix}
		+&+&0&0\\0&0&+&0\\0&0&-&+\\+&+&0&0
	\end{bmatrix}
$$
\end{minipage}
\hspace{0.05\textwidth}
\begin{minipage}{.45\textwidth}
\centering
\usetikzlibrary{decorations.markings}

\tikzset{
  vertex/.style={
    draw, circle, fill=black,
    inner sep=0pt, minimum size=0.15cm
  },
  edge/.style={
    thick,
    postaction={
      decorate,
      decoration={
        markings,
        mark=at position 0.75 with {\arrow{stealth}}
      }
    }
  },
  loopedge/.style={thick, ->}
}

\begin{tikzpicture}[scale=1, every node/.style={font=\small}]

\node[vertex,label=above:$1$]  (v1) at (0,0)   {};
\node[vertex,label=above:$2$] (v2) at (2,0)   {};
\node[vertex,label=below:$3$] (v3) at (2,-2)  {};
\node[vertex,label=below:$4$] (v4) at (0,-2)  {};

\draw[edge] (v1) -- node[above] {$+$} (v2);
\draw[edge] (v2) -- node[right] {$+$} (v3);
\draw[edge] (v3) -- node[below] {$+$} (v4);
\draw[edge] (v4) -- node[left]  {$+$} (v1);
\draw[edge] (v4) -- node[above] {$+$} (v2);

\draw[loopedge] (v1) edge[loop left] node {$+$} (v1);
\draw[loopedge] (v3) edge[loop right] node {$-$} (v3);

\end{tikzpicture}
\caption{The signed directed graph $D$.}
\label{figx5.10}

\end{minipage}
\end{figure}
Then $(1,1)(2,3,4)$, $(1,2,3,4)$ are two oppositely signed composite cycles in $D$ of length $4$, so $\mathcal{P}$ allows singularity but is not sign singular. Also, $D$ contains the cycle $(2,3,4)$ of length $3$. Moreover, $(1,1)$, $(2,3,4)$ are two odd cycles in $D$ of the same sign, while $(1,1)$, $(3,3)$ are two odd cycles in $D$ of opposite sign. Therefore, $\mathcal{P}$ satisfies all the conditions stated in Theorem~\ref{4.8l}. However, $$B=\begin{bmatrix}
		7&6&0&0\\0&0&1&0\\0&0&-5&1\\2&4&0&0
	\end{bmatrix}\in \mathcal{Q}(\mathcal{P})$$ and $B\in \mathcal{Q}(\mathcal{P})$ has all real eigenvalues, hence $\mathcal{P}\notin \Delta$.
	
\end{eg}

Eschenbach \cite{1993a}, gives some necessary and sufficient conditions for sign nonsingular sign patterns in $\pi$. They use the following lemma to establish the same.

\begin{theorem}[Theorem 1.8, \cite{1993a}]\label{th4.5}
	If $\mathcal{A}$ is an $n\times n$ ($n$ odd) sign nonsingular matrix, then $\mathcal{A}\in \pi$ if and only if $\mathcal{A}$ satisfies the following: \begin{itemize}
		\item[1.] all even cycles in $\mathcal{A}$ are nonpositive;
		\item[2.] if $\mathcal{A}$ contains a nonzero odd cycle $\gamma=a_{i_1i_2}a_{i_2i_3}\cdots a_{i_ki_1}$, then $\mathcal{A}[\{i_1,i_2,...,i_k\}^c]$ contains no nonzero odd cycles;
		\item[3.] $\mathcal{A}$ contains a nonzero even principal matchings of length $l>k$, where $k$ is the length of the smallest nonzero odd cycle; and
		\item[4.] all nonzero odd cycles in $\mathcal{A}$ have the same sign.
	\end{itemize}
\end{theorem}

The following theorem gives some necessary conditions for a sign nonsingular sign pattern to be in $\Delta$.

\begin{theorem} \label{4.9l}
	Let $\mathcal{P}\in \mathcal{Q}_n$ be a sign nonsingular sign pattern, whose underlying signed directed graph is $D$. If $\mathcal{P}\in \Delta$, then the following conditions hold.
    \begin{itemize}
		\item[1.] All composite cycles of length $n$ in $D$ have the same sign.
		\item[2.] There is either a negative even cycle or at least two nonzero vertex disjoint odd cycles in $D$.
       
		\item[3.]  If $D$ has a negative even cycle, then $D$ cannot have two nonzero odd cycles of the same sign that are vertex disjoint.
        \item[4.] If $D$ has a composite cycle $\Gamma$ of length $n$ that contains an odd cycle $\gamma$, then every odd cycle in $D(V(\gamma))$ has the same sign as $\gamma$.

	\end{itemize} 
\end{theorem}

\begin{proof}
	\begin{itemize}\item[1.] This is clearly true since $\mathcal{P}$ is sign nonsingular.
		
		\item[2.] Since $\mathcal{P}$ is sign nonsingular, there exists a composite cycle of length $n$, and since $\mathcal{P}\in \Delta$ by Lemma \ref{l3}, every composite cycle in $D$ of length $n$ must have a negative even cycle or two nonzero odd cycles which are vertex disjoint.

		\item[3.] Let $\gamma$ be a negative even cycle in $D$ and let $\gamma_1$, $\gamma_2$ be two nonzero vertex disjoint odd cycles of the same sign in $D$. Define the real matrix $B_\gamma(0)=[b_\gamma(0)_{ij}]$ as in \eqref{e1}, then the perturbed matrix $B_\gamma(\epsilon)=[b_\gamma(\epsilon)_{ij}]$, as defined in \eqref{e2}, is in $\mathcal{Q}(\mathcal{P})$ and has two nonzero real eigenvalues $\lambda$, $-\lambda$ close to $1,-1$ for some $\epsilon > 0$ sufficiently small. Since $\mathcal{P}\in \Delta$, so all the other eigenvalues of $B_\gamma(\epsilon)$ are nonreal and occur in conjugate pairs $\alpha_i$, $\bar{\alpha}_i$ for $i=1,2,...,\frac{n-2}{2}$. Hence, $$\det(B_\gamma(\epsilon))=\lambda  (-\lambda)  \alpha_1  \bar{\alpha}_1\cdots \alpha_{\frac{n-2}{2}}  \bar{\alpha}_{\frac{n-2}{2}}<0.$$
		
		Consider $\Gamma=\gamma_1\gamma_2$, and define the real matrix $B_{\Gamma}(0)=[b_{\Gamma}(0)_{ij}]$ as in \eqref{xe3}. 
         Then the perturbed matrix $B_{\Gamma}(\epsilon)$ as in \eqref{e2} with $\gamma$  replaced by $\Gamma$, has two nonzero real eigenvalues $\lambda_1$, $\lambda_2$ of the same sign close to $10$ or $-10$ and $10^2$ or $-10^2$ depending on the sign of $\gamma_1$, $\gamma_2$. Since $\mathcal{P}\in \Delta$ and $B_{\Gamma}(\epsilon)\in \mathcal{Q}(\mathcal{P})$, all the other eigenvalues of $B_{\Gamma}(\epsilon)$ are nonreal and occur in conjugate pairs. So,
		$$\det(B_{\Gamma}(\epsilon))=\lambda_1  \lambda_2  \alpha_1 \bar{\alpha}_1\cdots \alpha_{\frac{n-2}{2}} \bar{\alpha}_{\frac{n-2}{2}}>0.$$ 
		
		Since $\det(B_{\gamma}(\epsilon))<0$ and $\det(B_{\Gamma}(\epsilon))>0$, 
        which contradicts that $\mathcal{P}$ sign nonsingular.
		
	\item[4.] Assume that $D(V(\gamma))$ has an odd cycle $\gamma_1$ with sign opposite to $\gamma$. 
As $l(\Gamma)$ is even, there exists an odd cycle $\gamma_2$ in $D$ that is vertex disjoint with $\gamma$ and $\gamma_1$. Consider $\Gamma_1 = \gamma \gamma_2$ and $\Gamma_2 = \gamma_1 \gamma_2$. 
Since $\mathcal{P} \in \Delta$, by a similar argument as that in part~(3), there exist two real matrices $B_{\Gamma_1}(\epsilon), B_{\Gamma_2}(\epsilon) \in \mathcal{Q}(\mathcal{P})$ satisfying 
\[
\det(B_{\Gamma_1}(\epsilon)) = -\det(B_{\Gamma_2}(\epsilon)),
\]
which contradicts the assumption that $\mathcal{P}$ is sign nonsingular.

	\end{itemize}
\end{proof}	
However, the above conditions are not not sufficient for a sign pattern to be $2$-consistent.
\begin{eg}\rm
 Let $\mathcal{P}\in \mathcal{Q}_n$ be a sign pattern, whose underlying signed directed graph is $D$  given in Fig.\ref{figx5.17}.
		\begin{figure}[H]
\centering
\begin{minipage}{.45\textwidth}
\vspace{-1cm}
$$
\mathcal{P}=\begin{bmatrix}
		+&+&0&0\\0&0&+&0\\0&0&+&+\\-&+&0&0
	\end{bmatrix}
$$
\end{minipage}
\hspace{0.05\textwidth}
\begin{minipage}{.45\textwidth}
\centering 
\begin{tikzpicture}[scale=1, every node/.style={font=\small}]
\tikzset{
  vertex/.style={
    draw, circle, fill=black,
    inner sep=0pt, minimum size=0.15cm
  },
  edge/.style={
    thick,
    postaction={
      decorate,
      decoration={
        markings,
        mark=at position 0.75 with {\arrow{stealth}}
      }
    }
  },
  loopedge/.style={thick, ->} 
}

\node[vertex,label=above:$1$]  (v1) at (0,0)   {};
\node[vertex,label=above:$2$] (v2) at (2,0)   {};
\node[vertex,label=below:$3$] (v3) at (2,-2)  {};
\node[vertex,label=below:$4$] (v4) at (0,-2)  {};

\draw[edge] (v1) -- node[above] {$+$} (v2);
\draw[edge] (v2) -- node[right] {$+$} (v3);
\draw[edge] (v3) -- node[below] {$+$} (v4);
\draw[edge] (v4) -- node[left]  {$-$} (v1);
\draw[edge] (v4) -- node[above] {$+$} (v2);

\draw[loopedge] (v1) edge[loop left] node {$+$} (v1);
\draw[loopedge] (v3) edge[loop right] node {$+$} (v3);

\end{tikzpicture}
\caption{The signed directed graph $D$.}
\label{figx5.17}
\end{minipage}
\end{figure}
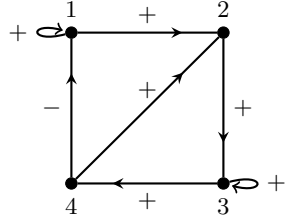 Clearly $\mathcal{P}$ is sign nonsingular. 
  Since all the composite $4$-cycles are positive. Note that $\mathcal{P}$ contains $2$ vertex disjoint odd cycles $(1,1), (2,3,4)$, and no negative even cycle. Also, $(1,1), (2,3,4)$ are only odd cycles which are contain a composite $4$ cycle, where $D((2,3,4))$ has a positive loop which is of the same sign as $(2,3,4)$.  So $\mathcal{P}$ satisfies all the conditions stated in Theorem \ref{4.9l}. However, the eigenvalues of 
	$$B=\begin{bmatrix}
		3&6&0&0\\0&0&1&0\\0&0&1&1\\-2&4&0&0
	\end{bmatrix}\in \mathcal{Q}(\mathcal{P})$$ are all purely imaginary, so $\mathcal{P}\notin\Delta$. 
\end{eg}

\section{Conclusion}\label{s51}

In this article, we investigated the consistency of sign pattern matrices. In particular, Theorem~\ref{th3.1} shows that if $\mathcal{P}$ is an irreducible, tridiagonal sign pattern with a $0$-diagonal of order at most $5$, then $\mathcal{P}$ is consistent if and only if it requires a unique inertia. This motivates the following conjecture.

\textbf{Conjecture.} If $\mathcal{P}$ is an irreducible, tridiagonal sign pattern with a $0$-diagonal, then $\mathcal{P}$ is consistent if and only if $\mathcal{P}$ requires a unique inertia.

In Theorems~\ref{th3.51n} and~\ref{th3.5n}, we show that if $\mathcal{P}$ is an irreducible, tridiagonal sign pattern with a $0$-diagonal such that the underlying signed undirected graph has a maximal signed path of length $1$ or $3$ not containing any leaf of $G$, then $\mathcal{P}$ is not consistent. We believe that the result can be extended to any maximal sign path not containing a leaf. Moreover, for such patterns if the only maximal signed odd length path includes a leaf then whether it is consistent remains an open question.

In Section~\ref{s5}, we obtain necessary conditions for irreducible sign patterns to require exactly two real eigenvalues. This naturally leads to the problem of characterizing this class $\Delta$, which provides scope for further research. 

\section{Acknowledgement} The research work of Partha Rana was supported by the Council of Scientific and Industrial Research (CSIR), India (File Number 09/731(0186)/2021-EMR-I).

	
\end{document}